\documentclass[11pts]{article}
\usepackage{graphicx}
\usepackage{amsmath}
\usepackage{amssymb}
\usepackage{amsthm}
\usepackage[usenames,dvipsnames]{xcolor}
\newtheorem{theorem}{Theorem}[section]
\newtheorem{corollary}[theorem]{Corollary}
\newtheorem{lemma}[theorem]{Lemma}
\newtheorem{proposition}[theorem]{Proposition}
\newtheorem{definition}[theorem]{Definition}
\newtheorem{remark}[theorem]{Remark}

\def\gm{\gamma}
\def\Gm{\Gamma}
\def\AA{\mathbf{A}}

\def\RR{\mathbb{R}}

\newcommand{\median}{\mathcal M}
\newcommand{\support}{\mathcal S}

\newcommand{\model}{\mathcal M}

\newcommand{\obs}{\mathbf y}
\newcommand{\x}{\mathbf x}

\newcommand{\noise}{\mathbf e}

\newcommand{\VV}{\mathcal V}
\newcommand{\UU}{\mathcal U}
\newcommand{\EE}{\mathcal E}
\newcommand{\K}{\mathcal K}
\newcommand{\NN}{[N]}
\newcommand{\supp}{\text{supp}}
\newcommand{\ngroups}{M}
\newcommand{\GG}{\mathfrak G}
\newcommand{\G}{\mathcal G}

\newcommand{\bigO}{\mathcal O}

\newcommand{\Real}{\mathbb{R}}

\newcommand{\A}{\mathcal{A}}

\newcommand{\bitem}{\begin{itemize}}
\newcommand{\eitem}{\end{itemize}}

\DeclareMathOperator*{\argmin}{argmin}

\newcommand{\beqn}{\begin{equation}}
\newcommand{\eeqn}{\end{equation}}
\newcommand{\balign}{\begin{align}}
\newcommand{\ealign}{\end{align}}

\def \dim {N}      

\begin{document}

\title{\Large Model-based Sketching and Recovery with Expanders\thanks{A preliminary version of this paper appeared in Proceedings of SODA 2014 \cite{bah2014model}. The current version includes an extended introduction, some previously omitted proof material and various minor corrections. This work was supported in part by the European Commission under Grant MIRG-268398, ERC Future Proof,  SNF 200021-132548, SNF 200021-146750 and SNF CRSII2-147633.}}
\author{Bubacarr Bah, Luca Baldassarre, Volkan Cevher \thanks{Laboratory for Information and Inference Systems (LIONS), \'Ecole Polytechnique F\'ed\'erale de Lausanne (EPFL), Switzerland ({\tt \{bubacarr.bah, luca.baldassarre, volkan.cevher\}@epfl.ch}).}}
\date{}

\maketitle


\begin{abstract} \small\baselineskip=9pt 
Linear sketching and recovery of sparse vectors with randomly constructed sparse matrices has numerous applications in several areas, including compressive sensing, data stream computing, graph sketching, and combinatorial group testing. This paper considers the same problem with the added twist that the sparse coefficients of the unknown vector exhibit further correlations as determined by a known sparsity model. 
We prove that exploiting model-based sparsity in recovery provably reduces the sketch size without sacrificing recovery quality. 
In this context, we present the model-expander iterative hard thresholding algorithm for recovering model sparse signals from linear sketches obtained via sparse adjacency matrices of expander graphs with rigorous performance guarantees. 
The main computational cost of our algorithm depends on the difficulty of projecting onto the model-sparse set. For the tree and group-based sparsity models we describe in this paper, such projections can be obtained in linear time. 
Finally, we provide numerical experiments to illustrate the theoretical results in action.
\normalsize
\end{abstract}
 
\section{Introduction} \label{sec:intro}

\subsection{Problem overview and statement}

In recent years, we have witnessed the advent of a new paradigm of succinct approximate representations of signals. A signal $\x \in \RR^N$ can be represented by the {\em measurements} or {\em sketch} $\AA\x$, for $\AA \in \RR^{m\times N}$, where $m \ll N$. The sketching is linear and dimensionality reducing, however it can still retain most of the information about the signal. These representations have been successfully applied in such diverse domains as data streaming \cite{muthukrishnan2005data,indyk2007sketching}, compressing sensing \cite{donoho2006compressed,candes2006robust}, graph sketching \cite{ahn2012graph,gilbert2004compressing} and even to breaking privacy of databases via aggregate queries \cite{dwork2007price}.

The standard approach in linear sketching assumes that $\x$ is $k$-sparse, i.e., only $k$ of its $N$ components are nonzero. We then find conditions on the sketch matrix $\AA$ and the recovery algorithm $\Delta$ to obtain an estimate $\widehat{\x} = \Delta(\AA\x + {\bf e})$, where ${\bf e}$ measures the perturbation of the linear model, that yields an error guarantee such as: 
\begin{align}
\label{eqn:cs_error_bnd}
\|\widehat{\x} - \x\|_p \leq C_1 \sigma_k(\x)_q + C_2\|{\bf e}\|_p,
\end{align}
where $\sigma_k(\x)_q = \min_{\scriptsize k-\mbox{sparse} ~\x'}\|\x - \x'\|_q,$ for some constants $C_1, C_2>0$ and $1\leq q \leq p \leq 2 $. We denote this type of error guarantees as $\ell_p/\ell_q$.
It has been established in CS that a randomly generated $\AA$ achieves the optimal sketch length of $\bigO\left(k\log(N/k)\right)$ and an approximate reconstruction, with error guarantee such as in  \eqref{eqn:cs_error_bnd}, of the original signal is possible using an appropriate algorithm, such as $\Delta$. This leads to two different probabilistic models for which \eqref{eqn:cs_error_bnd} holds.
The {\em for all} case where a randomly generated $\AA$ and recovery algorithm $\Delta$ satisfy \eqref{eqn:cs_error_bnd} {\em for all} $\x\in\RR^N$ with high probability; and the {\em for each} case where a randomly generated $\AA$, independent of $\x$, and recovery algorithm $\Delta$ satisfy \eqref{eqn:cs_error_bnd} {\em for each} given $\x\in\RR^N$ with high probability \cite{berinde2008combining}. 

In many cases, additional prior knowledge on the signal is available and can be encoded in a structure that models the relationships between the signal's sparse coefficients. The acquisition and recovery of signals with such {\em a priori} known structures, like tree and block structures is referred to as model-based compressed sensing (CS) \cite{baraniuk2010model,indyk2013model}. Model-based CS has the advantage of reducing the minimum sketch length for which a variant of \eqref{eqn:cs_error_bnd} holds. However, model-based CS with random \textit{dense} matrices require large storage space and  can lead to implementation difficulties. To overcome this computational drawback, this work, similarly to \cite{indyk2013model}, focuses on sparse measurement matrices that have the so called $\ell_1$-norm restricted isometry property (to be defined later), designed for model-based CS, and referred to as model-expanders. 

The recent work of \cite{indyk2013model} is the first to address model-based CS with sparse measurement matrices and in the {\em for all} case. They show that model-expanders exist when the structures (models) considered are binary trees and non-overlapping block structures and they provided bounds for the sketch sizes. They also propose a recovery algorithm that has exponential time complexity with respect to the dimension of the signal. 

\subsection{Main results}

In this paper, we propose an efficient polynomial time algorithm for model-based CS when the measurement matrices are model-expanders for sparsity models endowed with a tractable projection operator. We show that the proposed algorithm has $\ell_1/\ell_1$ theoretical guarantees. Our theoretical results are supported with numerical experiments. In addition, this work builds on \cite{indyk2013model} to extend the class of models for which we can show existence of model-expanders  that yield sampling bounds similar to those in \cite{indyk2013model}. Note that our results are for the {\em for all} case.

Formally, in order to recover $k$-model-sparse signals (to be defined in Section \ref{sec:prelim}), we propose the algorithm: {\em Model-expander iterative hard thresholding (MEIHT)}, which is an adaptation of the {\em iterative hard thresholding for expanders (EIHT)} algorithm \cite{foucart2013mathematical}.  The runtime of MEIHT is determined by the complexity of the {\em projection} onto a given model. In this paper we consider two sparsity models endowed with linear time projections via dynamic programs, adapted from \cite{baldassarre2013group}. Therefore, for these models, MEIHT has linear time complexity. 

MEIHT yields a $k$-model-sparse approximation $\widehat{\x}$, with a  {\em for all}, $\ell_1/\ell_1$ error guarantee as follows:
\begin{equation}
\label{eqn:mcs_error_bnd}
\|\widehat{\x} - {\bf x}\|_1 \leq C_1\sigma_{\model_k}(\x)_1 + C_2\|{\bf e}\|_1 \; ,
\end{equation}
where $\sigma_{\model_k}(\x)_1 := \min_{\scriptsize k-\mbox{model-sparse} ~\x'}\|\x - \x'\|_1$, for some constants $C_1,C_2 > 0$. This is a modification of \eqref{eqn:cs_error_bnd} that ensures that the approximations belong to the model considered. 

We consider two categories of structured signals: (a) $k$-rooted connected signals which have their $k$ nonzero components lie in a rooted connected tree \cite{baraniuk1994signal, baraniuk2010model,indyk2013model};  and (b) $k$-group sparse signals whose support is contained in the union of $k$ overlapping groups \cite{jacob2009group,rao2012signal}. Note that $k$-block sparse signals whose support consists of $k$ equal size non-overlapping blocks \cite{yuan2006model, stojnic2009reconstruction, baraniuk2010model,indyk2013model} are a special case of (b). We show that our algorithm can efficiently recover signals from these models and our numerical results support this. In addition, we generalize the results of \cite{indyk2013model} to $D$-ary tree models for (a) and to overlapping groups, instead of only block sparse signals. 

Our key contribution is introducing the first linear time algorithm for these models that resolves a key bottleneck in the application of model-expanders. Another contribution is the generalization of the results of \cite{indyk2013model} to cover a broader class of models. This is also significant for applications, since for instance we can consider wavelets quad-trees instead of just binary wavelet trees. In the case of tree-sparse models, our results provide an extension to the partial answer \cite{indyk2013model} to an open problem in the {\em Bertinoro workshop open problem list} \cite{indyk2011open}. Our extension is twofold: $(i)$ we considered a wider class of models, from binary to $D$-ary tree-sparse models for $D\geq 2$; $(ii)$ we proposed an efficient polynomial time algorithm.

A hallmark of our approach is that the algorithm can be used also for sparsity models that do not necessary lead to a reduction in sketch size. However, the algorithm always returns a signal from the given sparsity model. In many applications, this may be more important than reducing the sketch size, because it allows to interpret the solution with respect to the chosen sparsity structure even in the presence of perturbations. 

\subsection{Related work}

It has been observed that many signals have an underlying structure that can be exploited in compressed sensing, see for example \cite{duarte2005joint, fornasier2008recovery, eldar2009robust, stojnic2009reconstruction}, just to cite a few. Model-based compressed sensing \cite{baraniuk2010model, indyk2013model} was the terminology coined to refer to compressed sensing that exploits the structure of the signal in order to obtain approximations that satisfy \eqref{eqn:mcs_error_bnd} from fewer measurements than $\bigO\left(k\log\left(N/k\right)\right)$. For instance using dense sketching matrices \cite{baraniuk2010model} shows that we can have sketch lengths of $\bigO(k)$ and $\bigO\left(kg + k\log(N/(kg))\right)$ for the rooted-connected tree and the block structure models respectively - where $g$ is the size of the blocks. Recently, \cite{indyk2013model} used sparse sketching matrices to obtain $\bigO\left(k\log\left(N/k\right)/\log\log\left(N/k\right)\right)$ and $\bigO\left(kg\log(N)/\log(kg)\right)$ for the same models considered in \cite{baraniuk2010model}. The former work provided $\ell_2/\ell_1$ guarantees in the for all case with recovery algorithms typically of time complexity $\bigO\left(N\log(N)\right)$; while the latter gave $\ell_1/\ell_1$ guarantees in the for all case with an exponential time recovery algorithm. 

It was shown in \cite{indyk2011k} that the $\ell_2/\ell_1$ guarantees provided in \cite{baraniuk2010model} for block and tree models can be expressed into the $\ell_1/\ell_1$ framework, but it yields a super-constant approximation factor $C = \mathbf{\Theta}\left( \log(N) \right)$. Consequently, this generated interest into derivations that will avoid such super-constant. In fact, this was formally posed as an open problem by Poitr Indyk\footnote{See Question 15: Sparse Recovery for Tree Models.} in the {\em Bertinoro workshop open problem list} \cite{indyk2011open}. The authors of \cite{indyk2013model} provided a partial answer for the case when the trees are binary, but could not come up with an efficient recovery algorithm. This work extends the above results to $D$-ary trees for $D\geq 2$ and it proposes an efficient polynomial time algorithm - MEIHT.

It is well known that sparse matrices have computational advantages over their dense counterparts. Berinde et al.\ in \cite{berinde2008combining} introduced sparse binary, non-zero mean, sensing matrices that are adjacency matrices of expander graphs and they showed such matrices have the $\ell_1$-norm restricted isometry property (RIP-1). A binary matrix $\AA$ with $d$ nonzeros per column has RIP-1 \cite{berinde2008combining} if for all $k\mbox{-sparse vectors } \x\in\RR^N$, it satisfies
\begin{equation}
\label{eqn:rip1}
(1-\delta_{k})d\|\x\|_1 \leq \|\AA\x\|_1 \leq d\|\x\|_1,
\end{equation}
where $\delta_k$ is the restricted isometry constant (RIC) of order $k$. It is the smallest such constant over all sets of size $k$.
These sparse matrices have low storage complexity since they have only $d$, a small fraction, nonzero entries per column and they lead to much faster matrix products. 

Sparse recovery algorithms, for the generic sparse recovery problem with sparse non-zero mean matrices and {\em for all} guarantees, include $\ell_1$-minimization \cite{berinde2008combining} and combinatorially greedy algorithms: Sparse Matching Pursuit (SMP) \cite{berinde2008practical} and Expander Recovery Algorithm \cite{jafarpour2009efficient} (see also \cite{gilbert2010sparse,bah2012vanishingly}). However, \cite{indyk2013model} is the first to propose an exponential time algorithm for model-sparse signals using model expanders --- which they called model-based RIP-1 matrices --- in the {\em for all} case. We improve on this work by proposing the first polynomial time algorithm for model-sparse signals using model expander sensing matrices, when the sparsity model is endowed with a tractable projection operator.

The proposed algorithm involves projecting a signal onto the considered model.
Projections onto group-based models and rooted connected subtrees have been recently considered in \cite{baldassarre2013group} and \cite{cartis2013exact}. In the former paper, the authors showed that projections onto group-based models are equivalent to solving the combinatorial Weighted Maximum Coverage problem. Interestingly, there exist classes of group-based models that admit tractable projections that can be computed via dynamic programming in $\bigO(Nk)$ time, where $k$ is the maximum number of allowed groups in the support of the solution \cite{baldassarre2013group}. In \cite{cartis2013exact}, another dynamic program for rooted connected subtree projections has been proposed that achieves a computational complexity of $\bigO(Nk)$. The projection step in our algorithm uses a modification of the results of \cite{baldassarre2013group}. 

Typically the adjacency matrices of expander graphs are constructed probabilistically \cite{bassalygo1973complexity} leading to sparse sketching matrices with the desired {\em optimal} number of rows, $m = \bigO\left(k\log(N/k)/\epsilon\right)$ for $\epsilon > 0$. Deterministic construction of these objects with optimal parameters (graph degree $d$ and $m$) is still an open challenge. The current best {\em sub-optimal} deterministic construction is due to \cite{guruswami2009unbalanced} with $m = \bigO\left(d^2k^{1+\alpha}\right)$ for $\alpha > 0$. The recent work of \cite{indyk2013model} on model-based RIP-1 matrices (model expanders) also used a probabilistic construction. As an extension of this work, we show that it is possible to probabilistically construct more general model expanders. It is conceivable that {\em model-optimal} deterministic construction of these model expanders is possible, but it is not the focus of this work. We remark that the first work on model-based sparse matrices \cite{price2011efficient} considered signal models similar to those in \cite{indyk2013model} and achieved $\ell_2/\ell_2$ error guarantees, but only for the {\em for each} case.

\paragraph{Outline}
In the next section we set our notation and give definitions needed in the sequel. In Section \ref{sec:algo}, we present our algorithm and analyze its error guarantees. 
In Section \ref{sec:mdlexpanders}, we extend the existence of model-based expanders for $D$-ary rooted connected tree models and group-based models with overlaps. We provide numerical evidence in Section \ref{sec:emperics} and some concluding considerations in Section \ref{sec:concln}.
Finally, Section \ref{sec:proof} contains detailed proofs of the main results (Theorem \ref{thm:MEIHT_conv} and Corollary \ref{cor:MEIHT_conv}). 

\section{Preliminaries} \label{sec:prelim}

\subsection{Notation}

Throughout this paper, we denote scalars by lowercase letters (e.g. $k$), vectors by lowercase boldface letter (e.g., ${\bf x}$), sets by uppercase calligraphic letters (e.g., $\mathcal{S}$) and matrices by uppercase boldface letter (e.g. ${\bf A}$).
We use the shorthand $[N] := \{1, \ldots, N\}$.
Given a set $\mathcal{S} \subseteq [N]$, we denote its complement by $\bar{\mathcal{S}} := [N] \setminus \mathcal{S}$ and use $\x_\mathcal{S}$ for the restriction of $\x \in \Real^N$ onto $\mathcal{S}$, i.e., $(\x_\mathcal{S})_i = \x_i$ if $i \in \mathcal{S}$ and $0$ otherwise.
We use $|\mathcal{S}|$ to denote the cardinality of a set $\mathcal{S}$.
We denote the set of $k$-sparse vectors as $\Sigma_k$.

We denote bipartite graphs by $(\UU,\VV,\EE)$ where $\UU$ and $\VV$ are the set of left and right nodes, respectively, and $\EE$ is the set of edges.
A bipartite graph is left-regular if each left node has the same number of edges.
We call this number the degree of the graph.
Given a bipartite graph $(\UU,\VV,\EE)$ and a set $\K \subseteq \UU$, we use $\Gamma(\K)$ to indicate the set of {\em neighbors} of $\K$, that is the right nodes that are connected to the nodes in $\K$.
The $\ell_p$ norm of a vector ${\bf x} \in \Real^N$ is defined as $\|{\bf x}\|_p := \left ( \sum_{i=1}^N x_i^p \right )^{1/p}$.

\subsection{Model-based expanders} \label{sec:dfns}

In this paper, we consider sketching operators that are adjacency matrices of particular bipartite graphs (i.e., {\em model-based expanders}) that capture the structure of signals that belong to some pre-specified models.
We start by defining lossless expanders and describing models of structured signals before we define model-based expanders.

\begin{definition}[Lossless Expander \cite{capalbo2002randomness}]
\label{def:llexpander}
Let a graph $G=\left( [N],[m],\EE \right)$ be a left-regular bipartite graph with $N$ left (variable) nodes, $m$ right (check) nodes, a set of edges $\EE$ and left degree $d$.
If, for some $\epsilon \in (0,1/2)$ and any $\K \subset [N]$ of size $|\K|\leq k$, we have $|\Gamma(\K)| \geq (1-\epsilon)d|\K|$, then $G$ is referred to as a {\em $(k,d,\epsilon)$-lossless expander graph}.
\end{definition}

The property $|\Gamma(\K)| \geq (1-\epsilon)d|\K|$ is termed the {\em expansion} (or $\epsilon$-expansion) property and we refer to $\epsilon$ as the expansion coefficient. Typically, lossless expanders have $\epsilon \ll 1$, for instance $\epsilon <1/4$ \cite{jafarpour2009efficient}.

Informally, by {\em model-based expanders}, we mean that the set $\K$ in Definition \ref{def:llexpander} has some structure, so that the expansion property must hold for a reduced number of index sets $\K$.
Specifically, we restrict $\K$ to two models: a hierarchical model, where the elements of $[N]$ are organized in a {\em tree}, and a group model, where the elements of $[N]$ are grouped in $M$ {\em overlapping groups}.

\begin{definition}[$k$-rooted connected subtree]
\label{def:tree}
Given a $D$-ary tree $\mathcal{T}$ for an integer $D\geq 2$, an index set $\K$ is a {\em $k$-rooted connected subtree} of $\mathcal{T}$,  if it contains at most $k$ elements from $\mathcal{T}$ and for each element in $\K$, all its ancestors with respect to $\mathcal{T}$ are also in $\K$. This implies that the root of $\mathcal{T}$ is always in $\K$.
\end{definition}

\begin{definition}[$k$-tree sparse model and set]
Given a $D$-ary tree $\mathcal{T}$ ($D \geq 2$), we define the collection of all $k$-rooted connected subtrees of $\mathcal{T}$ as the {\em $k$-tree sparse model} $\mathcal{T}_k$.
A set $\support\subseteq [N]$ is {\em $k$-tree sparse} (i.e., $\mathcal{T}_k$-sparse) if $\support \subseteq \K$ for some $\K\in\mathcal{T}_k$.
\end{definition}

The second model is based on groups of variables that should be selected together.

\begin{definition}[group structure]
A {\em group structure} $\GG = \{\G_1, \ldots, \G_\ngroups\}$ is a collection of $M$ index sets, named {\em groups}, with $\G_j \subseteq \NN$ and $|\G_j| = g_j$ for $ 1 \leq j \leq \ngroups$ and $\bigcup_{\G \in \mathfrak{G}} \G = \NN$.
\end{definition}

We say two groups $\G_i$ and $\G_j$ overlap if $\mathop{\G_i \cap \G_j \neq \emptyset}$.
Given a group structure $\GG$, it is possible to define a graph --- the {\em group-graph} --- whose nodes are the groups $\{\G_1, \ldots, \G_\ngroups\}$ and whose edges connect nodes $i$ and $j$ if $\G_i$ and $\G_j$ overlap.
If the group-graph induced by $\GG$ does not contain loops, we say that $\GG$ is a {\em loopless overlapping group} structure.
We consider this structure in the sequel, because it turns out that projecting onto such group models can be done in linear time.

\begin{definition}[$k$-group sparse model and set]
Given a group structure $\GG$, the {\em $k$-group sparse} model $\GG_k$ is defined as the collection of all sets that are the union of at most $k$ groups from $\GG$.
A set $\support \in [N]$ is {\em $k$-group sparse} (i.e., $\GG_k$-sparse) if $\support \subseteq \K$ for some $\K\in\GG_k$.
\end{definition}

We use $\mathcal{M}_k$ to jointly refer to the $\mathcal{T}_k$ and $\GG_k$ models and $\model = \{\model_1, \model_2, \ldots \}$ to represent the {\em model class}.
Next we define model-sparse vectors/signals.
\begin{definition}[model-sparse vectors]
\label{def:mdlsparse}
Given a model $\model_k$, a vector ${\bf x} \in \Real^N$ is {\em model-sparse}, denoted as $\model_k$-sparse, if $\supp({\bf x}) \subseteq \K$ for some $\K \in\model_k$.
\end{definition}

The following concept of model expander hinges on the previous definitions.
\begin{definition}[Model-expander]
\label{def:mdlexpander}
Let a graph $G=\left( [N],[m],\EE \right)$ be a left-regular bipartite graph with $N$ left (variable) nodes, $m$ right (check) nodes, a set of edges $\EE$ and left degree $d$.
If, for any $\epsilon \in (0,1/2)$ and any $\support \subseteq \K$ with $\K \in \mathcal{M}_k$, we have $|\Gamma(\support)| \geq (1-\epsilon)d|\support|$, then $G$ is called a {\em $(k,d,\epsilon)$-model expander graph}.
\end{definition}

We interchangeably refer to the adjacency matrices of model-expander graphs as model-expanders.
As noted in \cite{indyk2013model}, it is straightforward to see that model-expanders satisfy a variant of the RIP-1, \cite{berinde2008combining}. Precisely, we say a model-expander, $\AA$ satisfies {\em model} RIP-1 if for all $\model_k\mbox{-sparse vectors } \x\in\RR^N$
\begin{equation}
\label{eqn:mrip1}
(1-\delta_{\model_k})d\|\x\|_1 \leq \|A\x\|_1 \leq d\|\x\|_1.
\end{equation}
where $\delta_{\model_k}$ is the model RIC of order $k$.
Henceforth, we also refer to model-expander matrices as model RIP-1 (denoted as $\model_k^{\epsilon}$-RIP-1) matrices as in \cite{indyk2013model}.

\subsection{Motivation} \label{sec:app}

Block models have direct applications in graph sketching, group testing and micro-arrays data analysis. Furthermore, these models and overlapping groups models have been advocated for several applications that include detecting genetic pathways for diagnosis or prognosis of cancer \cite{jacob2009group}, background-subtraction in videos \cite{baraniuk2010model}, mental state decoding via neuroimaging \cite{gramfort2009improving}, to mention just a few.
Loopless pairwise overlapping groups models have been identified in \cite{baldassarre2013group} as a class of group models that allows to find group-sparse signal approximations in polynomial time. An example of such structure is given by defining groups over a wavelet tree containing a node and all its direct children. These groups will overlap pairwisely on only one element and the corresponding group graph does not contain loops. They can be used to approximate signals which have significant features at different locations and at different scales.

On the other hand, rooted connected tree models arise in many applications ranging from image compression and denoising via wavelet coefficients to bioinformatics, where one observes  hierarchical organization of gene networks, to deep learning architectures with hierarchies of latent variables, to topic models that naturally present a hierarchy (or hierarchies) of topics (see, for example, \cite{jenatton2011proximal} for a more extensive list of applications and references). Another application of the rooted connected tree models is in Earth Mover Distance (EMD) sparse recovery \cite{indyk2011k}, where the pyramid transform can be used to convert the EMD recovery to a rooted connected sparse recovery with sparse binary matrices. Furthermore, recovery schemes with $\ell_1/\ell_1$ guarantees work well for EMD sparse recovery applications for both binary tree-sparse and wavelet quad-tree-sparse models. A typical application of EMD is in astronomical imaging of sparse point clouds where EMD sparse approximation schemes will approximately identify locations and weights of the clouds.

\section{Model based recovery with expanders} \label{sec:algo}

\subsection{Algorithm overview} \label{algover}

The model-expander iterative hard thresholding (MEIHT) algorithm proposed is a hybrid adaptation of the {\em expander iterative hard thresholding} (EIHT) algorithm proposed in \cite{foucart2013mathematical} (which is a modification of the {\em sparse matching pursuit} (SMP) algorithm \cite{berinde2008practical}) and projections onto the considered models using the results in \cite{baldassarre2013group} and \cite{cartis2013exact}. Let the compressive measurements be given as $\obs = \AA \x + \noise$ for a signal of interest $\x\in\RR^N$ where the noise vector $\noise \in \RR^m$ characterizes the measurement error. Typically initializing ${\bf x}^0 = {\bf 0} \in \RR^N$, the SMP iterates the following 
\begin{equation}
\label{eqn:SMP}
{\bf x}^{n+1} = H_k \big{\{} {\bf x}^n + H_{2k}\left[\mathcal{M} \left( \obs - {\bf Ax}^n \right) \right]\big{\}}
\end{equation}
where $H_s({\bf u})$ is the hard thresholding operator which retains the $s$ largest in magnitude components of ${\bf u}$ and sets the rest to zero and $\mathcal{M}$ is the {\em median operator} which is nonlinear and defined componentwise on a vector, say ${\bf u}\in \RR^m$, as follows
\begin{equation}
\label{eqn:median}
\left[\mathcal{M}({\bf u})\right]_i := \mbox{median}\left[u_j, j \in \Gamma(i) \right] ~ \mbox{for} ~ i\in [N].
\end{equation}
SMP was modified to become the {\em expander iterative hard thresholding} (EIHT) algorithm in \cite{foucart2013mathematical}. Precisely, the EIHT usually initializes ${\bf x}^0 = {\bf 0} \in \RR^N$ and iterates
\begin{equation}
\label{eqn:EIHT}
{\bf x}^{n+1} = H_k \left[ {\bf x}^n + \mathcal{M} \left( \obs - {\bf Ax}^n \right) \right].
\end{equation}

The stability and robustness to noise of EIHT was shown in \cite{foucart2013mathematical}. We adapt EIHT and call it the {\em model expander iterative hard thresholding} (MEIHT) algorithm for the purpose of sparse recovery with model based expanders.  In MEIHT, we replace the hard thresholding operator, $H_s({\bf b})$, by the projector $\mathcal{P}_{\model_s}({\bf b})$, which projects the vector ${\bf b}$ onto our model of size $s$, $\model_s$, with respect to the $\ell_1$ norm (see Section \ref{sec:project} for more details). The pseudo-code of MEIHT is given in Algorithm 1 below. 
Note that the proposed algorithm always returns a signal from the given sparsity model.
\begin{table}[h]
\centering
\begin{tabular}{l}
\hline
 \textbf{Algorithm 1} Model-Expander IHT (MEIHT)\\
 \hline
 \textbf{Input:} Data: ${\bf A}$ and $\obs$ and a bound of $\|\noise\|_1$\\
 \textbf{Output:} $k$-sparse approximation $\widehat{\x}$\\
 \hline
 \textbf{Initialization:}\\
    \quad Set ${\bf x}^0=0$ \\
 \textbf{Iteration:} While some {\em stopping criterion} is not true \\
    \quad Compute ${\bf x}^{n+1} = \mathcal{P}_{\model_k} \left[ {\bf x}^n + \mathcal{M} \left( \obs - {\bf Ax}^n \right) \right]$ \\
 \textbf{Return} $\widehat{\x} = {\bf x}^{n+1}$ \\
 \hline
\end{tabular}
\label{tab:MEIHT_pcode}
\end{table}

\begin{remark}
SMP, EIHT and MEIHT could be regarded as projected gradient descent algorithms \cite{jorge1999numerical}. This is because the median operator behaves approximately like the adjoint of $\AA$ \cite{foucart2013mathematical,berinde2008practical}. Essentially, these algorithms do a gradient descent by applying the median and then project onto the feasible set using the hard threshold operator or the projector onto the model.
\end{remark}

\subsection{Runtime complexity of MEIHT} \label{sec:comp}
The complexity of MEIHT depends on the cost of the median operation and the cost of the projection onto a sparsity  model. It turns out that in the settings with model expander measurement matrices the projection on the models is costlier than the median operation. Therefore, for the models we consider, the runtime complexity of MEIHT is linear. A detailed discussion on the projection step and its cost is given in Section \ref{sec:project}. The following proposition formally states the complexity of MEIHT for the models considered.
\begin{proposition}
\label{pro:comp}
The runtime of MEIHT is $\bigO\left(knN\right)$ and $\bigO\left(M^2kn + nN\right)$ for the models $\mathcal{T}_k$ and $\GG_k$ respectively, where $k$ is the model order, $n$ is the number of iterations, $M$ is the number of groups and $N$ is the dimension of the signal.
\end{proposition}
Recall that $k$ being the model order means that signals in $\mathcal{T}_k$ are $k$-sparse while signals in $\GG_k$ are $k$-group sparse. It is clear from the proposition that the runtime is linear in $N$ for both models.

\begin{proof}
The algorithm has linear convergence (see below) and hence a fixed, say $n$, number of iterations are sufficient to achieve an approximate solution.
At each iteration MEIHT performs $N$ median operations of $d$ elements which involve sorting with $\bigO\left(d\log d\right)$ complexity and a projection step that costs $\bigO\left(kN\right)$ for the $\mathcal{T}_k$ model and $\bigO\left(M^2k + N\right)$ for the $\GG_k$ model. The projection step dominates since  $d$ is small, to be precise $d = \bigO\left(\log(N)/\log\log(N/k)\right)$ and $d = \bigO\left(\log(N)/\log\left(kg_{\max}\right)\right)$ for the former and latter models respectively, see Section \ref{sec:mdlexpanders}. 
Therefore the runtime of MEIHT is a product of the number of iterations and the complexity of the projections. 
\end{proof}

\subsection{Convergence of MEIHT} \label{sec:conv}
The convergence analysis of MEIHT follows similar arguments as the proof of the linear time convergence of EIHT given in \cite{foucart2013mathematical}. The key difference between the analysis of EIHT and MEIHT is the projections $H_k(\cdot)$ and $\mathcal{P}_{\model_k}(\cdot)$ respectively. We leverage and adapt recent results in \cite{baldassarre2013group} to perform the projection $\mathcal{P}_{\model_k}(\cdot)$ exactly and efficiently in linear time. The projections are analysed in Section  \ref{sec:project}. Theorem \ref{thm:MEIHT_conv} and Corollary \ref{cor:MEIHT_conv} below bound the error of the output of the MEIHT algorithm.

\begin{theorem}
\label{thm:MEIHT_conv}
Let $\mathcal{M}_k$ be a model and let $\mathcal{S}$ be $\mathcal{M}_k$-sparse. Consider ${\bf A} \in \{0,1\}^{m\times N}$ to be the adjacency matrix of a model-expander with $\epsilon_{\model_{3k}} < 1/12$. For any ${\bf x} \in \RR^N$ and ${\bf e} \in \RR^m$, the sequence of updates $\left({\bf x}^n\right)$ of {\em MEIHT} with $\obs = {\bf Ax} + {\bf e}$ satisfies, for any $n\geq 0$,
\begin{equation}
\label{eqn:MEIHT_error}
 \|{\bf x}^n - {\bf x}_{\mathcal{S}}\|_1 \leq \alpha^n\|{\bf x}^0 - {\bf x}_{\mathcal{S}}\|_1 + \left(1-\alpha^n\right)\beta\|{\bf Ax}_{\bar{\mathcal{S}}} + {\bf e}\|_1,
\end{equation}
where $\alpha < 1$ and $\beta$ depends only on $\epsilon_{\model_{3k}}$. 
\end{theorem}

\begin{corollary}
\label{cor:MEIHT_conv}
If the sequence of updates $\left({\bf x}^n\right)$ converges to $\widehat{\x}$ as $n\rightarrow \infty$, then 
\begin{equation}
\label{eqn:MEIHT_error_cor}
\|\widehat{\x} - {\bf x}\|_1 \leq C_1\sigma_{\model_k}(\x)_1 + C_2\|{\bf e}\|_1,
\end{equation}
for some constants $C_1,C_2>0$ depending only on $\epsilon_{\model_{3k}}$.
\end{corollary}

\noindent \textbf{Sketch proof of Theorem \ref{thm:MEIHT_conv} \& Corollary \ref{cor:MEIHT_conv}}
The first step of the proof consists in bounding $\|\left[ \mathcal{M} \left({\bf Ax}_{\mathcal{S}} + {\bf e}\right) - {\bf x} \right]_{\mathcal{S}}\|_1$. 
This is done inductively over the set $\mathcal{S}$ and key ingredients to this proof are properties of the median operator and the expansion property of expander graphs.
We obtain that 
\begin{equation}
\label{eqn:lem_median0}
 \|\left[ \mathcal{M} \left({\bf Ax}_{\mathcal{S}} + {\bf e}\right) - {\bf x} \right]_{\mathcal{S}}\|_1 \leq \frac{4\epsilon_{\model_{s}}}{1-4\epsilon_{\model_{s}}}\|{\bf x}_{\mathcal{S}}\|_1  + \frac{2}{(1-4\epsilon_{\model_{s}}) d}\|{\bf e}_{\Gamma(\mathcal{S})}\|_1,
\end{equation}
where $\mathcal{S} \in \mathcal{M}_s$-sparse for $|\mathcal{S}|\leq s$, $\epsilon_{\model_{s}}$ is the expansion coefficient of our model and $\Gamma(\mathcal{S})$ is the set of neighbours of $\mathcal{S}$. 
Now, assuming exact projections onto our model (in the $\ell_1$ norm), we use the above bound with the triangle inequality to have the following upper bound
\begin{equation}
\label{eqn:MEIHT_error0}
 \|{\bf x}^{n+1} - {\bf x}_{\mathcal{S}}\|_1 \leq \alpha\|{\bf x}^n - {\bf x}_{\mathcal{S}}\|_1 + (1-\alpha)\beta\|{\bf Ax}_{\bar{\mathcal{S}}} + {\bf e}\|_1 ,
\end{equation}
for some constants $\alpha<1$ and $\beta>0$ which depend only on $\epsilon_{\model_{3s}}$. This result relies on the fact that our models have the {\em nestedness property} discussed in Section \ref{sec:nested}. By induction, this bounds leads to \eqref{eqn:MEIHT_error}.

Corollary \ref{cor:MEIHT_conv} easily follows from Theorem \ref{thm:MEIHT_conv} by taking the limits of \eqref{eqn:MEIHT_error} as $n\rightarrow \infty$ and using a property of model RIP-1 matrices, which is also a property of the underlying expander graphs. We provide the detailed proofs of Theorem \ref{thm:MEIHT_conv} and Corollary \ref{cor:MEIHT_conv} in Section \ref{sec:proof}.

\subsection{Nestedness property of sparsity models} \label{sec:nested}

The RIP-1 requirement in traditional CS for the basic $k$-sparse case holds for sums and differences of sparse signals (Minkowski sums of signal supports). In model-based recovery, however, the Minkowski sums of sparsity models from a given model class (e.g., rooted connected tree models or group models) may not necessarily belong to the same model class and, as a result, model RIP-1 may not hold for such Minkowski sums \cite{baraniuk2010model}. For instance, the Minkowski sum, $\model_{k} \oplus \model_{k} \oplus \model_{k}$, of model orders of size $k$ may not be in the model $\model_{3k}$. Consequently, the model RIC of this Minkowski sum is not necessarily $\delta_{\model_{3k}}$ which is proportionalÊ to $\epsilon_{3k}$. However, if a given model has what we refer to as the {\em nestedness property}, defined below, the Minkowski sum, $\model_{k} \oplus \model_{k} \oplus \model_{k}$ is in the model class and the model RIC of the Minkowski sum is $\delta_{\model_{3k}}$.

In particular, the convergence proof of our algorithm relies on the fact that the models we consider are {\em nested}, in other words we require that Minkowski sums of signals in a  model remain in the model. This is equivalent to asking the model to have the {\em nestedness property} defined below which is similar to the {\em nested approximation property} (NAP) introduced in \cite{baraniuk2010model}.
\begin{definition}[Nestedness property]
\label{def:nested}
A model class $\model = \{ \model_1, \model_2, \ldots, \}$ has the {\em nestedness property} if, for any $\support \in \model_k$ and $\support' \in \model_{k'}$, with $\model_{k}, \model_{k'} \in \model$, we have $\support \cup \support' \in \model_{k+k'}$. 
\end{definition}
It is important to note that the models we consider in this work, rooted connected tree models and overlapping group models, have the nestedness property. In other words, the model-expanders expand on models of order $k,2k,3k$, and so on, simultaneously. For instance, the union of of the supports of two $\mathcal{T}_k$-sparse signals is the support of a $\mathcal{T}_{2k}$-sparse signal, similarly for $\GG_k$-group sparse signals.

Not all sparsity models have the nestedness property. An example is the dispersive model used to described neuronal spikes proposed in \cite{hegde2009compressive}, where the model-sparse signal is defined as a sparse signal whose nonzero coefficients must be separated by a fixed number of zero coefficients. Summing two such signals can yield a signal whose nonzero coefficients are not separated enough and hence does not belong to the desired sparsity model.

\subsection{Projections} \label{sec:project} 

The projections for the rooted connected tree model and the loopless overlapping groups model can be computed in linear time via dynamic programming, leveraging the results in \cite{baldassarre2013group} and \cite{cartis2013exact}, where the authors considered projections in the $\ell_2$ norm.

Given a general model $\mathcal{M}_k$, we define the projections in the $\ell_1$ norm as
\begin{equation}
\label{eq:model_proj}
\mathcal{P}_{\mathcal{M}_k}(\x) \in \argmin\limits_{{\bf z}: \supp({\bf z}) \in \mathcal{M}_k} \|{\bf x} - {\bf z}\|_1 .
\end{equation}
We can reformulate \eqref{eq:model_proj} as a discrete optimization problem over sets in $\mathcal{M}_k$. Indeed, we have
\begin{align}
\min_{{\bf z}: \supp({\bf z}) \in \mathcal{M}_k} \|{\bf x} - {\bf z}\|_1&  = \min_{\mathcal{S} \in \mathcal{M}_k} \min_{{\bf z} : \supp({\bf z}) \subseteq \mathcal{S}} \|{\bf x} - {\bf z}\|_1 \nonumber \\
& = \max_{\mathcal{S} \in \mathcal{M}_k} \max_{{\bf z} :  \supp({\bf z}) \subseteq \mathcal{S}} \|\x\|_1- \|{\bf x} - {\bf z}\|_1 \nonumber \\
& = \max_{\mathcal{S} \in \mathcal{M}_k} \|\x_\mathcal{S}\|_1 \label{eq:model_proj_discrete} .
\end{align}
Therefore, finding the projection in the $\ell_1$ norm of a signal $\x$ onto a model $\mathcal{M}_k$ amounts to finding a set in $\mathcal{M}_k$ that maximizes the ``cover" of $\x$ in absolute value. Tractable models $\mathcal{M}_k$ are models for which it is possible to solve \eqref{eq:model_proj_discrete} in polynomial time, leading to an overall polynomial time complexity for the proposed recovery algorithm, see Section \ref{sec:comp}.

Following \cite{baldassarre2013group}, we can show that projections in $\ell_1$ norm onto general group sparse models, which include the rooted connected tree model, can be reformulated as the combinatorial Weighted Maximum Coverage (WMC) problem. Since the WMC is NP-hard in general, these projections are NP-hard too. However, \cite{baldassarre2013group} identified group structures that lead to linear time algorithms. These structures correspond to the models we consider: loopless overlapping group models and rooted connected tree models.

For projections onto a group model $\GG_k$, we can reformulate \eqref{eq:model_proj_discrete} as
\begin{equation}
\label{eq:group_proj}
\max_{\scriptsize \begin{array}{c} \mathcal{S} \subseteq \GG\\ |\mathcal{S}| \leq k\\\mathcal{I} = \bigcup_{\G \in \mathcal{S}} \G\end{array}} \|{\bf x}_\mathcal{I}\|_1 \; .
\end{equation}
Let us now define the binary matrix ${\bf A}^\mathfrak{G} \in \mathbb{B}^{N \times M}$,
$$
{\bf A}^\mathfrak{G}_{ij} = \bigg \{ \begin{array}{lc} 1, & \text{if}~i \in \G_j; \\ 0, & \text{otherwise.} \end{array} \; 
$$
This matrix fully encodes the group structure and allows to rewrite \eqref{eq:group_proj} as the following Weighted Maximum Coverage problem
\begin{equation}
\label{eq:WMC}
\max\limits_{\boldsymbol{\omega} \in \mathbb{B}^\ngroups,~{\bf y} \in \mathbb{B}^\dim} \left \{ \sum_{i=1}^N y_i |x_i| : {\bf A}^\mathfrak{G} \boldsymbol{\omega} \geq {\bf y},  \sum_{j=1}^\ngroups \omega_j \leq k \right \},
\end{equation}
where $\boldsymbol{\omega}$ and $\bf y$ are binary variables that specify which groups and which variables are selected, respectively. The constraint $\A^\mathfrak{G} \boldsymbol{\omega} \geq {\bf y}$ makes sure that for every selected variable at least one group that contains it is selected, while $\sum_{j=1}^\ngroups \omega_j \leq k$ limits the number of selected groups to at most $k$. Note that the weights for the WMC are given by the absolute values of the coefficients of $\bf x$.

Two recently proposed dynamic programs \cite{baldassarre2013group, cartis2013exact} compute the projections in the $\ell_2$ norm for the loopless overlapping groups model and the rooted connected tree model.
The dynamic programs gradually explore a graph defined by the considered model and recursively compute the optimal solution.
In order to adapt the dynamic programs to projections in the $\ell_1$ norm, it is sufficient to replace the weights of the nodes of the graphs by using the absolute values of the components of ${\bf x}$ instead of the square value.

For the $k$-tree model, the graph is the given tree where each node correspond to a variable and has weight equal to the variable absolute value.
The tree is explored from the leaves upward, updating at each node a table containing the optimal selection of $1$ to $k$ nodes from its subtrees.
More specifically, consider being at a particular level in the tree. For each node in this level, we store a table that contains the optimal values for choosing $1$ to $k$ connected elements from its subtree. When we move one level up, the table for a node in the new level is computed by comparing the tables of its subtrees, considering one subtree at a time, from right to left.
The computational complexity is $\bigO(ND^2k)$, where $D$ is the tree degree ($D = 2$ for binary trees).

For the case of loopless overlapping groups, \cite{baldassarre2013group} considered the group-graph induced by the group structure, for which a node corresponds to an entire group of variables.
The weight associated to each node is evaluated dynamically during the exploration of the group-graph: it is the sum of the weights of the variables included in that node that do not belong to an already selected group.
By defining an appropriate graph exploration rule, the dynamic program has time complexity equal to $\bigO(M^2k + N)$.

\begin{remark}[Generalization]
In \cite{baldassarre2013group}, the overlapping groups model has been generalized to within-group sparsity with an overall sparsity budget $K$, together with the group budget $k$.
A dynamic program solves the projection \eqref{eq:model_proj} exactly in $\bigO(M^2K^2k + N\log K)$ operations for this model.
\end{remark}

\section{Existence of model-expanders} \label{sec:mdlexpanders}

For practical applications of model expanders we need to be able to construct these objects. The key goal is to construct them with as small parameters ($d$ and $m$) as possible. Bounds, both lower and upper, on $m$ for model-expanders, for binary tree-sparse and block-sparse models, are derived in \cite{indyk2013model}. The derivation of the lower bounds depended mainly on RIP-1, since the sparsification technique they employed relies on the RIP-1 of these matrices. Therefore, the extension of these two models to $D$-ary tree-sparse and fixed overlapping tree-sparse models respectively, which we consider in this manuscript, also have these lower bounds. We skip the explicit derivation because this will be identical to the derivation in \cite{indyk2013model}.

However, since the upper bounds involve the enumeration of the cardinality of the models which is different from \cite{indyk2013model}, we explicitly show the derivations below. In essence, we will show the existence of model expanders by random (probabilistic) construction. Precisely, for every vertex, $i$, in the set of left vertices $[N]$, i.e. $i\in [N]$, we sample with replacement $d$ vertices in the set of right vertices, $[m]$ and then we connect these $d$ nodes to $i$. Each $d$-subset of $[m]$ are sampled uniformly and independently of other sets. This leads to a standard tail bound whose proof uses a Chernoff bound argument. However, in our derivation we will use the following variant of this standard tail bound introduced and proved \cite{buhrman2002bitvectors}.

\begin{lemma}{\cite{buhrman2002bitvectors}}
\label{lem:mdl_existence}
There exist constants $C>1$ and $\mu > 0$ such that, whenever $m \geq Cdt/\epsilon$, for any $T\subseteq [N]$ with $|T|=t$ one has
\begin{equation*}
\label{eqn:mdl_existence}
\hbox{Prob}\left[|\left\{j\in [m] ~\left| ~\exists  i \in T, e_{ij} \in \EE  \right\}\right| < (1-\epsilon)dt\right] \leq \left(\mu \cdot \frac{\epsilon m}{dt}\right)^{-\epsilon dt}.
\end{equation*}
\end{lemma}

\subsection{$D$-ary tree model}
$D$-ary tree model-expanders or $\mathcal{T}_k^{\epsilon}$-RIP-1 matrices are $m\times N$ sparse binary matrices with $d$ ones per column. The relation between the $[N]$ indices of the columns of these matrices is modeled by a $D$-ary tree. Theorem \ref{thm:mdlexpander_rc} states the existence and the sizes of the parameters of these matrices.
\begin{theorem}
\label{thm:mdlexpander_rc}
For $\epsilon \in (0,1/2)$ and $k=\omega\left(\log N\right)$ there exists a $\mathcal{T}_k^{\epsilon}$-RIP-1 matrix with
\begin{equation}
\label{eqn:mdlx_param}
d = \bigO\left(\frac{\log\left(N/k\right)}{\epsilon\log\log\left(N/k\right)}\right) \quad and \quad m = \bigO\left(\frac{dk}{\epsilon}\right).
\end{equation}
\end{theorem}

\begin{proof}
For a model $\mathcal{T}_k \subseteq \Sigma_k$ and $t\in [k]$ denote $\mathcal{T}_k$-sparse sets of size $t$ as $\mathcal{T}_{k,t}$ and denote the number of $D$-ary rooted connected trees with $k$ nodes as $T_k$. We use the following estimate of the number of sets $\mathcal{T}_{k,t}$.
\begin{lemma}
\label{lem:num_trees}
For all $t\in [k]$, we have
\begin{equation*}
\left| \mathcal{T}_{k,t}\right| \leq \min \left[ T_k \cdot \binom{k}{t},\binom{N}{t}\right] , \quad \mbox{where} \quad T_k = \frac{1}{(D-1)k+1}\binom{Dk}{k}.
\end{equation*}
\end{lemma}
\begin{proof} It is sufficient to know that the $T_k$  are the Pfaff-Fuss-Catalan numbers or $k$-Raney numbers \cite[Note 12]{lang2009combinatorial}, which enumerates the total number of ordered, rooted $D$-ary trees of cardinality $k$.
\end{proof}

Using the fact that $\binom{x}{y}$ is bounded above by $\left(\frac{ex}{y}\right)^y$, we have that for all $t\in [k]$, $\left|\mathcal{T}_{k,t}\right|$ is upper bounded by
\begin{equation}
\label{eqn:num_trees}
\min \left[ \frac{(eD)^k}{(D-1)k+1} \left(\frac{ek}{t}\right)^t,\left(\frac{eN}{t}\right)^t\right].
\end{equation}
Taking a union bound over the number of $\mathcal{T}_{k,t}$ we see that the probability of Lemma \ref{lem:mdl_existence} goes to zero with $N$ if
$$\forall t\in[k] \quad  \left| \mathcal{T}_{k,t} \right| \cdot \left(\mu \cdot \frac{\epsilon m}{dt}\right)^{-\epsilon dt} \leq \frac{1}{N}, $$
where $m$ satisfies the bound in the lemma. Let us define $t^* := \frac{k\log(eD) - \log(Dk - D + 1)}{\log\left(N/k\right)}$. By simple algebra we see that at $t = t^*$ the two quantities in the square brackets of \eqref{eqn:num_trees} are equal; and $\forall t\in \left[1,t^*\right]$ the second quantity is smaller while $\forall t\in \left[t^*,k\right]$  the first quantity is smaller. Therefore, we split the domain of $t$ into two intervals and prove that $\forall t\in[1,t^*] $
\begin{equation}
\label{eqn:split_intervals1}
\left(\frac{eN}{t}\right)^t \left(\mu \cdot \frac{\epsilon m}{dt}\right)^{-\epsilon dt} \leq \frac{1}{N},
\end{equation}
and that $\forall t\in[t^*,k]$
\begin{equation}
\label{eqn:split_intervals2}
\frac{(eD)^k}{(D-1)k+1} \left(\frac{ek}{t}\right)^t \left(\mu \cdot \frac{\epsilon m}{dt}\right)^{-\epsilon dt} \leq \frac{1}{N}.
\end{equation}
In accordance with Theorem \ref{thm:mdlexpander_rc}, we let the left degree $d = C'\log\left(N/k\right)/\left[\epsilon\log\log\left(N/k\right)\right]$ and $m = C''dk/\epsilon$. Note that the left hand side of Inequalities \eqref{eqn:split_intervals1} and \eqref{eqn:split_intervals2} are log-convex whenever $d>1/\epsilon$. We therefore need to check the conditions only at the end points of the intervals, i.e. for $t = 1, t^*, k$. For $t=1$, Inequality \eqref{eqn:split_intervals1} becomes
$$eN\left(C''\mu k\right)^{- C'\log\left(N/k\right)/\log\log\left(N/k\right)} \leq 1/N. $$
This inequality holds for sufficiently large $C'$ and $C''$ given that $k = \omega(\log N)$. Thus \eqref{eqn:split_intervals1} holds for $t = 1$.
When $t= t^*$, we have the left hand side of inequality \eqref{eqn:split_intervals1} equal to the left hand side of inequality \eqref{eqn:split_intervals2}. Substituting the value of $t^*$ and those of $m$ and $d$ in the left hand side of \eqref{eqn:split_intervals1} gives the product of the following two terms.
\begin{equation*}
\left(\frac{eN\log\left(N/k\right)}{\eta}\right)^{\frac{k\log(eD) - \log(Dk - D + 1)}{\log\left(N/k\right)}}
\times \left(\frac{C''\mu k \log\left(N/k\right)}{\eta}\right)^{-C' \frac{k\log(eD) - \log(Dk - D + 1)}{\log\log\left(N/k\right)}} ,
\end{equation*}
where $\eta =  k\log(eD) - \log(Dk - D + 1)$.
The above quantity is less than $1/N$ for sufficiently large $C'$ and $C''$ given that $k = \omega(\log N)$. Thus \eqref{eqn:split_intervals1} and \eqref{eqn:split_intervals2} hold for $t = t^*$.
For $t = k$, with the values of $m$ and $d$, the left hand side of \eqref{eqn:split_intervals1} becomes
$$\frac{(e^2D)^k}{(D-1)k+1} \left(C''\mu\right)^{- C'k \log\left(N/k\right)/\log\log\left(N/k\right)}$$
This quantity is far less than $1/N$ for sufficiently large $C'$ and $C''$ given that $k = \omega(\log N)$. Thus \eqref{eqn:split_intervals2} hold for $t = k$ and this concludes the proof.
\end{proof}

\subsection{Overlapping group model}
Fixed overlapping group model-expanders or $\GG_k^{\epsilon}$-RIP-1 matrices are also $m\times N$ sparse binary matrices with $d$ ones per column and the relationship between the indices of their columns is modeled by an overlapping group structure. Theorem \ref{thm:mdlexpander_lg} states the existence and the sizes of the parameters of these matrices.
\begin{theorem}
\label{thm:mdlexpander_lg}
Let the total number of groups of the model be $M$ and the maximum size of groups be $g_{\max}$. For any $\epsilon \in (0,1/2)$, if $M\geq 2$, $g_{\max} =\omega\left(\log N\right)$, and $N > kg_{\max}$, then there exist a $\GG_k^{\epsilon}$-RIP-1 matrix, $\AA \in \RR^{m\times N}$ with $d$ nonzero entries per column, such that
\begin{equation}
\label{eqn:mdlx_param_lg}
d = \bigO\left(\frac{\log(N)}{\epsilon \log\left(kg_{\max}\right)}\right) \quad \mbox{and} \quad
m = \bigO\left(\frac{dkg_{\max}}{\epsilon}\right).
\end{equation}
\end{theorem}
Intuitively, we see that the condition $N>kg_{\max}$ is necessary in order to not have just one large group; while the condition $g_{\max} = \omega(\log N)$ means that $g_{\max}$, the maximum size of groups, cannot be too small.

\begin{proof}
Let $\GG_{k,t}$ denote $\GG_k$-sparse sets with sparsity $t$. We use the following estimate of the number of sets $\GG_{k,t}$. Then we have $\forall ~t\in [kg_{\max}]$
\begin{equation}
\label{eqn:num_groupsets}
\left| \GG_{k,t}\right| \leq \min \left [\binom{M}{k}\binom{k g_{\max}}{t}, \binom{N}{t} \right].
\end{equation}

Since $\binom{x}{y}$ is bounded above by $\left(\frac{ex}{y}\right)^y$, we have for all $t\in [kg_{\max}]$
\begin{equation}
\left| \GG_{k,t}\right| \leq \min \left [\left(\frac{eM}{k}\right)^k \left(\frac{ek g_{\max}}{t}\right)^t, \left(\frac{eN}{t}\right)^t \right] .
\end{equation}

Using a union bound over the sets $\GG_{k,t}$, we see that the probability of Lemma \ref{lem:mdl_existence} goes to zero with $N$ if
$$
\forall t \in[kg_{\max}], \quad  \left|  \GG_{k,t} \right| \cdot \left(\mu \cdot \frac{\epsilon m}{dt}\right)^{-\epsilon dt} \leq \frac{1}{N},
$$
where $m$ satisfies the bound in the lemma.
By simple algebra we see that at $t = t^* = \frac{k\log\left(eM/k\right)}{\log\left(N/(kg_{\max})\right)}$ the two quantities in the square brackets of \eqref{eqn:num_groupsets} are equal; and $\forall t\in \left[1,t^*\right]$ the second quantity is smaller, while $\forall t\in \left[t^*,kg_{\max}\right]$ the first quantity is smaller. Therefore, we split the domain of $t$ into two intervals and prove that $\forall t\in[1,t^*]$
\begin{equation}
\label{eqn:split_intervals1_group}
\left(\frac{eN}{t}\right)^t \left(\mu \cdot \frac{\epsilon m}{dt}\right)^{-\epsilon dt} \leq \frac{1}{N},
\end{equation}
and that $\forall t\in[t^*,kg_{\max}]$
\begin{equation}
\label{eqn:split_intervals2_group}
\left(\frac{eM}{k}\right)^k \left(\frac{ekg_{\max}}{t}\right)^t \left(\mu \cdot \frac{\epsilon m}{dt}\right)^{-\epsilon dt} \leq \frac{1}{N}.
\end{equation}
In accordance with Theorem \ref{thm:mdlexpander_lg}, we let the left degree $d = C'\log\left(N\right)/\left[\epsilon\log\left(kg_{\max}\right)\right]$ and $m = C''dkg_{\max}/\epsilon$. The left hand side of Inequalities \eqref{eqn:split_intervals1_group} and \eqref{eqn:split_intervals2_group} are log-convex whenever $d>1/\epsilon$. We therefore need to check the conditions only at the end points of the intervals, i.e. for $t = 1, t^*, kg_{\max}$. For $t=1$, Inequality \eqref{eqn:split_intervals1_group} becomes
$$eN\left(C''\mu kg_{\max}\right)^{- C'\log\left(N\right)/\log\left(kg_{\max}\right)} \leq 1/N. $$
This is true for sufficiently large $C'$ and $C''$ given that $g_{\max} = \omega(\log N)$. Thus \eqref{eqn:split_intervals1_group} holds for $t = 1$.
When $t= t^*$, we have the left hand side of inequality \eqref{eqn:split_intervals1_group} equal to the left hand side of inequality \eqref{eqn:split_intervals2_group}. Substituting $t^*$ for $t$ and the values of $m$ and $d$ in the left hand side of \eqref{eqn:split_intervals1_group} leads to
$$\left(\frac{eN}{t^*}\right)^{t^*}\left(\frac{C''\mu kg_{\max}}{t^*}\right)^{- C't^*\log\left(N\right)/\log\left(kg_{\max}\right)} \leq \frac{1}{N}. $$
The above inequality is true for sufficiently large $C'$ and $C''$ given that $g_{\max} = \omega(\log N)$. Thus \eqref{eqn:split_intervals1_group} and \eqref{eqn:split_intervals2_group} hold for $t = t^*$.
For $t = k$, with the values of $m$ and $d$, the left hand side of \eqref{eqn:split_intervals1_group} becomes
$$\left(\frac{e^2Mg_{\max}}{k}\right)^k \left(C''\mu g_{\max}\right)^{-C'k\log(N)/\log\left(kg_{\max}\right)} \leq \frac{1}{N}.$$
This is also true for sufficiently large $C'$ and $C''$ given that $g_{\max} = \omega(\log N)$. Thus \eqref{eqn:split_intervals2_group} hold for $t = k$ and this concludes the proof.

The extra condition on $N$ in the theorem can be explicitly derived if we impose that $t^*>0$ which leads to $N>kg_{\max}$. But this condition is implicitly assumed in the above derivation.
\end{proof}

\begin{remark}
The block structured model is a special case of the overlapping group model, when the groups do not overlap. Usually in this setup the blocks are assumed to be of the same length which is equivalent to setting $g_1 = g_2 = \cdots = g_M = N/M =: g$. The above analysis on overlapping groups carries through to this particular case in a very straight forward way. The fixed overlapping group model is another special case of the more general overlapping groups model proven above.
\end{remark}

\section{Experimental Results} \label{sec:emperics}


\subsection{Block sparsity}
In the first experiment, we consider recovering a block-sparse signal, where the $M$ equal-sized blocks define a partition of the set $\{1, \ldots, N\}$, so that the block size is given by $g = N/M$.

We fix the number of active groups, $k = 5$, and let the group size $g$ grow as $\log(N)$ to satisfy the assumptions of our bound for group-sparse recover \eqref{eqn:mdlx_param_lg}. We also have that the number of groups $M$ must grow as $N/\log(N)$.
Specifically, we set $N \in \{2^7, 2^8, \ldots, 2^{13}\}$, $M = \lfloor N/\log_2(N) \rfloor$, and $g = \lfloor N/M \rfloor$.
We vary the number of samples $m$ and compute the relative recovery error $\|\widehat{\x} - \x\|_1/\|\x\|_1$, where $\widehat{\x}$ is the signal estimated either with the EIHT or the MEIHT algorithms and $\x$ is the true signal.
For each problem size $N$, we repeat the experiment $50$ times. For each repetition, we randomly select $k$ blocks from $M$ and set the components of the true signal $\x$ as identical and independent draws from a standard Gaussian distribution.
The other components are set to zero.
For each number of samples, we randomly draw a sparse sketching matrix $\AA \in \Real^{m \times N}$ with $d=\lfloor 2\log(N)/\log(kg) \rfloor$ ones per column in order to satisfy the condition in \eqref{eqn:mdlx_param_lg}.
The columns of $\AA$ are normalized to have unitary $\ell_1$ norm.
We then create the measurement vector as $\obs = \AA \x$, which we use to recover $\x$ using either the EIHT or the MEIHT algorithm.
We record the minimum number of samples $m^*$ that yields a median relative recovery error below $10^{-5}$.

In the top panel of Figure \ref{fig:block}, we plot $m^*$ as a function of $\log_2(N)$, showing strong empirical evidence for our theoretical results.
Indeed, in the setting of this experiment, our bound predicts that $m^*$ for MEIHT should scale as $\log(N)^2/\log(C_1\log(N))$, where $C_1$ is a positive constant.
Instead the bound for standard expanders indicates that $m^*$ for EIHT should scale as $\log(N)\log(N/C_2\log(N))$, where $C_2$ is a positive constant.
The plot shows that empirically the minimum number of samples for good recovery grows faster for EIHT than for MEIHT.

\begin{figure}[h!]
\centering
\includegraphics[width=0.46\textwidth]{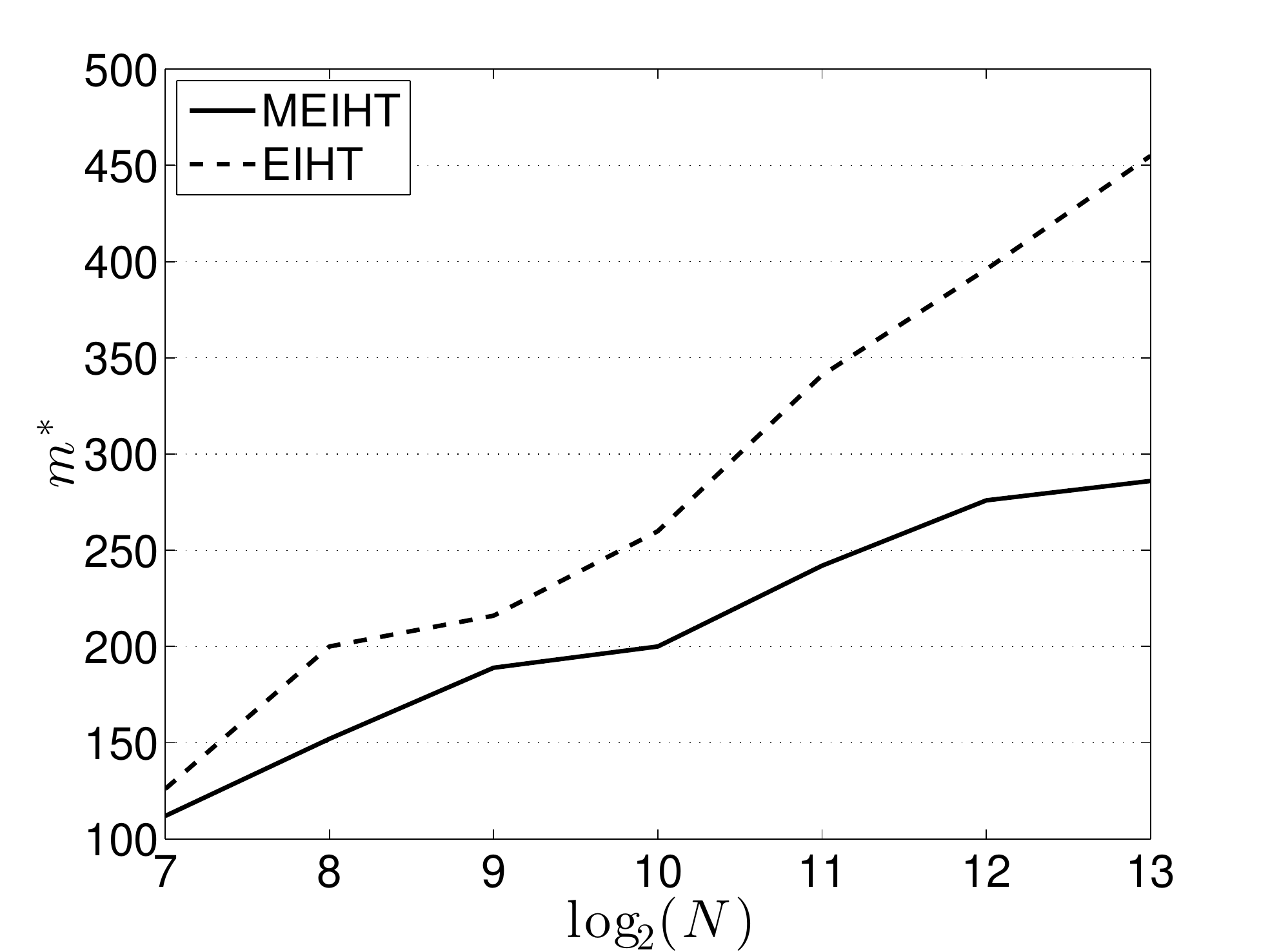}
\includegraphics[width=0.46\textwidth]{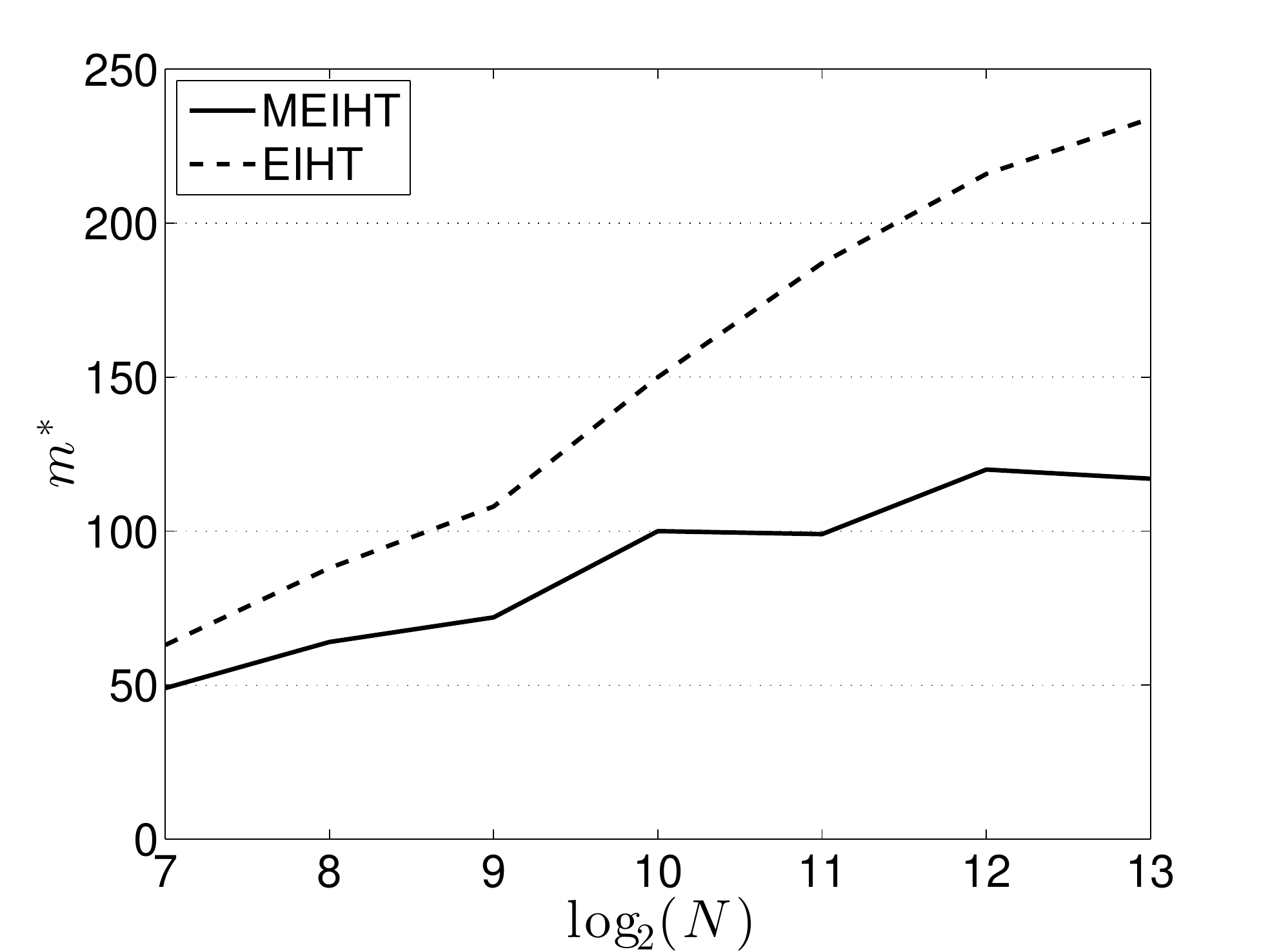}
\caption{\label{fig:block} Model-based signal recovery example. {\em Left panel:} block-sparse model. The original signal is block-sparse (see text for details on how it is generated). The plot shows the minimum number of samples required to reach a median relative recovery error in the $\ell_1$ norm below $10^{-5}$ for the EIHT and the MEIHT algorithms as the signal ambient dimension is increased. {\em Right panel:} rooted connected subtree model. The original signal is sparse and its support is a rooted connected subtree. The plot shows the minimum number of samples required to reach a median relative recovery error in the $\ell_1$ norm below $10^{-5}$ for the EIHT and MEIHT algorithms as the signal ambient dimension is increased. Note that the x-axis in both plots reports the logarithm of the ambient dimension.}
\end{figure} 

\subsection{Tree sparsity}
\label{sec:exp_tree}
In the second experiment, we try to recover a signal $\x \in \Real^N$ whose nonzero coefficients lie in a rooted connected subtree of a regular binary tree.

We set the number of nonzero coefficients $k = \lfloor 2\log_2(N) \rfloor$ as required by Theorem \ref{thm:mdlexpander_rc}, while their values are randomly drawn from a standard Gaussian distribution.
We randomly draw a sparse sketching matrix $\AA \in \Real^{m \times N}$ with $d= \lfloor 2.5 \frac{\log\left(N/k\right)}{\log\log\left(N/k\right)}\rfloor$ ones per column in order to satisfy the condition in \eqref{eqn:mdlx_param}.
The columns of $\AA$ are normalized to have unitary $\ell_1$ norm.
We use the EIHT and MEIHT algorithms to obtain an approximation $\widehat{\x}$ from the linear sketch $\AA \x$.
For each value of $N \in \{2^7, 2^8, \ldots, 2^{13}\}$, we vary the sketch length, $m \in [2k,~10k\log_2N]$, and compute an approximation $\widehat{\x}$ with both algorithms, recording the relative recovery error $\|\widehat{\x} - \x\|_1/\|\x\|_1$.
We repeat the experiment $50$ times with a different draw of the signal $\x$ and compute the median of the relative recovery errors for both algorithms.
Finally, we record the minimum number of measurements, $m^*$, that yields a median relative recovery error less than $10^{-5}$.

The results are presented in the bottom panel of Figure \ref{fig:block} and show a strong accordance to the theoretical results of the previous section.
Indeed, in the setting of this experiment, our bound predicts that $m^*$ for MEIHT should scale as $\frac{\log(N)\log(N/(C\log(N)))}{\log\log(N/C\log(N))}$, where $C$ is a positive constant.
Instead, the bound for standard expanders indicates that $m^*$ for EIHT should scale as $\log(N)\log(\frac{N}{C\log(N)})$.
The plot shows that in practice the minimum number of samples for good recovery grows faster for EIHT than for MEIHT, matching the theoretical bounds.

\section{Conclusions} \label{sec:concln}

In this paper, we consider signals that belong to some general structured sparsity models, such as rooted connected trees or overlapping groups models. 
We focus on the problem of recovering such signals from noisy low-dimensional linear sketches obtained from sparse measurement matrices.
We present the first linear time recovery algorithm with $\ell_1/\ell_1$ recovery guarantees for the models we consider. 
Our results rely on: $(i)$ identifying a special class of sparse matrices that are the adjacency matrices of model-based expanders and which have the model $\ell_1$ norm restricted isometry property; $(ii)$ performing exact projections onto the sparsity models by adapting the dynamic programs presented in \cite{baldassarre2013group}. 
We show that by exploiting structured sparsity priors we are able to reduce the sketch length that guarantees recovery and our experimental results confirm these findings. In addition, we derive probabilistic constructions of a more general class of model expanders.

We conclude with a surprising result that reveals an interesting research question. We repeat the experiment in Section \ref{sec:exp_tree}, using sparse rooted connected tree signals with fixed sparsity, $k=16$, and sparse sketching matrices with a fixed number of ones per column (corresponding to $d=6$) in contrast to what required by the theory. Figure \ref{fig:surprise} shows that in this setting, the proposed algorithm seems to yield a constant sample complexity, comparable to the one predicted for dense sketching matrices. It may be possible that the better bound for dense sketching matrices may be achievable in the {\em for each} case also by sparse ones, bringing even greater efficiency advantages.

\begin{figure}[h]
\centering
\includegraphics[width=0.46\textwidth]{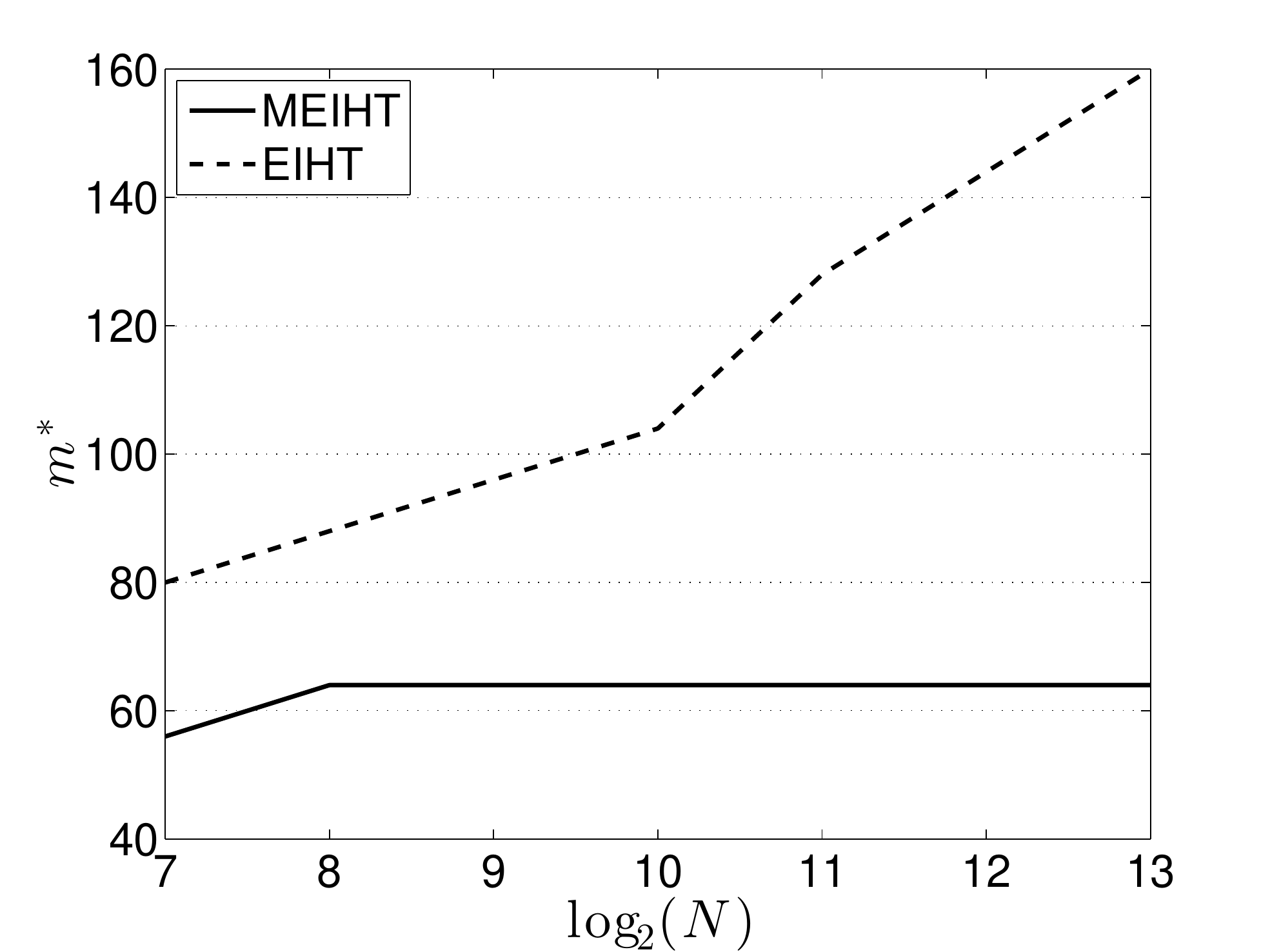}
\caption{\label{fig:surprise}A surprising result. Model-based signal recovery for the rooted connected tree model. The original signal is $k=16$ sparse and its support is a rooted connected subtree. The model expander is generated with $d = 6$ ones per column. The plot shows the minimum number of samples required to reach a median relative recovery error in the $\ell_1$ norm below $10^{-5}$ for the EIHT and MEIHT algorithms as the signal ambient dimension is increased. Note that the x-axis reports the logarithm of the ambient dimension.}
\end{figure}

\section{Proofs}\label{sec:proof}

Before we proceed with the proof of Theorem \ref{thm:MEIHT_conv}, we present an essential lemma. This lemma is similar to Lemma 13.16 in \cite{foucart2013mathematical} which holds only for the basic $k$-sparse model. Here, we generalize this lemma to cover the models we consider in this paper using similar proof techniques to that of \cite{foucart2013mathematical}.

We first start with additional notations and definitions that are required in this section. Recall that $\AA$ is the adjacency matrix of a $\left(k,d,\epsilon\right)$-model expander, $G = \left([N],[m],\mathcal{E}\right)$. Let $\support \subseteq [N]$, then $\Gamma'(\support)$ refers to the set of vertices connected to exactly one vertex in $\support$, while $\Gamma''(\support)$ refers to the set of vertices connected to more than one vertex in $\support$. Note that $\Gamma''(\support)$ is referred to as the {\em collision set} in the RIP-1 literature. The variable $e_{ij}$ indicates the edge from vertex $i$ to vertex $j$ while $e_i$ denotes the $i^{\mathrm{th}}$ component of the vector ${\bf e}$.

\subsection{Key lemma for the proof of Theorem \ref{thm:MEIHT_conv}}

Now, we state the lemma which is the key ingredient to the proof of Theorem \ref{thm:MEIHT_conv}. This lemma is about the sub-problem of our algorithm involving the median operation and it shows that the median operator approximately inverts the action of ${\bf A}$ on sparse vectors, i.e., it behaves more-or-less like the adjoint of $\AA$. For clarity, we provide the proof only for the $k$-tree sparse model, $\mathcal{T}_k$. The proof for the $k$-group  sparse model, $\mathfrak{G}_k$, easily follows {\em mutatis mutandis}.
\begin{lemma}
\label{lem:median}
Let ${\bf A} \in \{0,1\}^{m\times N}$ be the adjacency matrix of a model expander with an expansion coefficient of order $k$, $\epsilon_{\model_{k}}$, satisfying $4\epsilon_{\model_{k}} d < d + 1$. If $\mathcal{S} \subseteq [N]$ is $\mathcal{M}_k$-sparse, then for all ${\bf x} \in \RR^N$ and ${\bf e} \in \RR^m$
\begin{equation}
\label{eqn:lem_median}
\|\left[ \mathcal{M} \left({\bf Ax}_{\mathcal{S}} + {\bf e}\right) - {\bf x} \right]_{\mathcal{S}}\|_1 
 \leq \frac{4\epsilon_{\model_{k}}}{1-4\epsilon_{\model_{k}}}\|{\bf x}_{\mathcal{S}}\|_1 + \frac{2}{(1-4\epsilon_{\model_{k}}) d}\|{\bf e}_{\Gamma(\mathcal{S})}\|_1.
\end{equation}
\end{lemma}

The proof of this lemma requires the following proposition which is a property of $d$-left regular bipartite graphs like $(k,d,\epsilon)$-lossless and $(k,d,\epsilon)$-model expander graphs, which is equivalent to the RIP-1 condition \cite{berinde2008combining}. Remember that a $d$-left regular bipartite graph is a bipartite graph with each left vertex having degree $d$.
\begin{proposition}
\label{pro:unique_nbr}
Given a $d$-left regular bipartite graph with expansion coefficient $\epsilon$, if $\support$ is a set of $s$ left indices, then the set, $\Gm'(\support)$ of right vertices connected to exactly one left vertex in $\support$ has size $$|\Gm'(\support)| \geq (1-2\epsilon)ds.$$
\end{proposition}

\begin{proof} \textbf{Proof of Proposition \ref{pro:unique_nbr}}
By the expansion property of the graph we have $|\Gm(\support)| \geq (1-\epsilon)d|\support|$ and by the definition of the sets $\Gm'(\cdot)$ and $\Gm''(\cdot)$ we have $|\Gm(\support)| = |\Gm'(\support)| + |\Gm''(\support)|$ which therefore means
\begin{equation}
\label{eqn:num_nbrs1}
 |\Gm'(\support)| + |\Gm''(\support)| \geq (1-\epsilon)d|\support|.
\end{equation}
Now let the set of edges between $\support$ and $\Gm(\support)$ be denoted by $\mathcal{E}(\support,\Gm(\support))$. Hence $|\mathcal{E}(\support,\Gm(\support))|$ is the number of edges between $\support$ and $\Gm(\support)$ and by enumerating these edges we have
\begin{equation}
\label{eqn:num_nbrs2}
|\Gm'(\support)| + 2|\Gm''(\support)| \leq |\mathcal{E}(\support,\Gm(\support))| = d|\support|.
\end{equation}
Solving \eqref{eqn:num_nbrs1} and \eqref{eqn:num_nbrs2} simultaneously gives the bound $|\Gm''(\support)| \leq \epsilon d|\support|$ and thus $|\Gm'(\support)| \geq (1-2\epsilon)d|\support|$. Substituting $|\support|$ by $s$ completes the proof.
\end{proof}

\begin{proof}\textbf{Proof of Lemma \ref{lem:median}}
For a permutation $\pi:[m]\rightarrow [m]$ such that $b_{\pi(1)} \geq b_{\pi(2)} \geq \ldots \geq b_{\pi(m)}$, we can rewrite the median operation of \eqref{eqn:median} as
$$\left[\mathcal{M}({\bf u})\right]_i := q_{2}\left[u_j, j \in \Gamma(i) \right] ~~\mbox{for} ~~{\bf u}\in \RR^m ~~\mbox{and} ~~i\in [N]$$ where $q_2$ is a shorthand for the {\em second quantile (median) operator} that is $q_2\left[b_1,\ldots,b_m\right] = b_{\pi\left(\lceil m/2\rceil\right)}$.

Using this shorthand for the median, the left hand side of \eqref{eqn:lem_median} can be written in the following way:
\begin{align*}
\|\left[ \mathcal{M} \left({\bf Ax}_{\mathcal{S}} + {\bf e}\right) - {\bf x} \right]_{\support}\|_1 & = \sum_{i\in \support} \left|q_2 \left[ \left({\bf Ax}_{\support} + {\bf e}\right)_j, ~j\in \Gm(i) \right] - x_i\right| \\
& = \sum_{i\in \support} \left|q_2 \left[ \left({\bf Ax}_{\support}\right)_j + e_j - x_i, ~j\in \Gm(i) \right]\right|.
 \end{align*}
 The above can be bound as in \eqref{eqn:lmm_lhs2} below using the following property of the median operator, which is true of more general quantile operators:
$$\left|q_2\left[b_1,\ldots,b_m\right]\right| \leq q_2\left[\left|b_1\right|,\ldots,\left|b_m\right|\right].$$
Therefore, we have
 \begin{equation}
 \label{eqn:lmm_lhs2}
 \begin{split}
\|\left[ \mathcal{M} \left({\bf Ax}_{\mathcal{S}} + {\bf e}\right) - {\bf x} \right]_{\support}\|_1 & \leq \sum_{i\in \support}q_2 \left[\left| \left({\bf Ax}_{\support}\right)_j + e_j - x_i\right|, ~j\in \Gm(i) \right] \\
& = \sum_{i\in \support}q_2 \Big[\Big| \mathop{\sum_{l\in \support\backslash \{i\}}}_{e_{lj}\in \mathcal{E}} x_l + e_j\Big|, ~j\in \Gm(i) \Big],
\end{split}
\end{equation}
where the last equality is due to the property of the underlying expander graph of $\AA$, say $G$, and remember $e_{lj}\in \mathcal{E}$ denotes an edge from node $l$ to node $j$ being in $\mathcal{E}$, the set of edges of $G$. Now, we proceed by induction on the cardinality of $\support$ i.e., $s = |\support|$ to show that the right hand side of \eqref{eqn:lem_median} upper bounds the right hand side of \eqref{eqn:lmm_lhs2}.

{\bf Base case:} If $\support = \{i\}$ that is $|\support|=s=1$, then there is no $l \in \support \backslash \{i\}$. Using the following property of the median operator, which is also true for more general quantile operators:
\begin{equation}
\label{eqn:quantile2}
q_2\left[b_1,\ldots,b_m\right] \leq \frac{2\left(b_1+\cdots+b_m\right)}{m}, ~ \mbox{if}  ~ \forall j ~ b_j \geq 0,
\end{equation}
we have the following stronger bound.
\begin{equation}
\label{eqn:lmm_lhs_bnd1}
\sum_{i\in \support}q_2 \Big[\Big| \mathop{\sum_{l\in \support\backslash \{i\}}}_{e_{lj}\in \mathcal{E}} x_l + e_j\Big|, ~j\in \Gm(i) \Big] = q_2\left[\left|e_j\right|, j \in \Gm(i)\right] \leq \frac{2}{d}\|\noise_{\Gm(i)}\|_1.
\end{equation}

{\bf Inductive step:} We assume that the induction statement holds for $\support$ with $|\support| = s-1$ ($s\geq 2$) and proceed to show that it also holds for $\support$ with $|\support| = s$. By the pigeonhole principle, given a $\left(s,d,\epsilon_{\mathcal{T}_{s}}\right)$-model expander, Proposition \ref{pro:unique_nbr} implies that there exists a variable node $i^*\in \support$ which is connected uniquely to at least $(1-2\epsilon_{\mathcal{T}_{s}})d$ check nodes. We will refer to this node as the {\em unique neighbor} index. Any set $\support\in \mathcal{T}_s$-sparse (with $|\support| = s$) can be partitioned into $\{i^*\}$ and $\support\backslash\{i^*\}$. We will bound  \eqref{eqn:lmm_lhs2} for (a) when $i\in \{i^*\}$, that is $i=i^*$, and (b) when $i\in \mathcal{R} :=\support\backslash\{i^*\}$. Summing the bounds from (a) and (b) will result in an upper bound on \eqref{eqn:lmm_lhs2} for the case $|\support| = s$ and this completes the proof.

(a) By the definition of $q_2$, there exist $\lceil d/2\rceil$ distinct $j_1,\ldots j_{\lceil d/2\rceil} \in \Gm(i^*)$ ($i^*$ is the unique neighbor node) such that, for all $h\in \left\{1,\ldots,\lceil d/2\rceil\right\}$,
\begin{equation}
\label{eqn:lmm_lhs_bnds1}
q_2 \Big[\Big| \mathop{\sum_{l\in \support\backslash \{i^*\}}}_{e_{lj}\in \mathcal{E}} x_l + e_j\Big|, ~j\in \Gm(i^*) \Big] \leq \Big| \mathop{\sum_{l\in \support\backslash \{i^*\}}}_{e_{lj_h}\in \mathcal{E}} x_l + e_{j_h}\Big|, ~j_h\in \Gm(i^*).
\end{equation}
Note that there are no $l\in \support \backslash \{i^*\}$ with $e_{lj_h}\in \mathcal{E}$ since $i^*$ is the unique neighbor index. Therefore, we can also bound \eqref{eqn:lmm_lhs_bnds1} by the mean over the $j_h$. Let us define $\gm := |\Gm'(i^*)| \geq (1-2\epsilon_{\mathcal{T}_{s}})d$. There are at most $d-\gm \leq 2\epsilon_{\mathcal{T}_{s}} d$ vertices in the collision set $\Gm''(i^*)$ and there are at least $\lceil d/2 \rceil - (d-\gm) \leq \lceil d/2 \rceil - 2\epsilon_{\mathcal{T}_{s}} d$ elements among the $j_h$ that are in $\Gm'(i^*)$. We use this fact to upper bound the mean over the $j_h$ elements which in turn upper bounds \eqref{eqn:lmm_lhs_bnds1} as follows
\begin{equation}
\label{eqn:lmm_lhs_bnds2}
\begin{split}
q_2 \Big[\Big| \mathop{\sum_{l\in \support\backslash \{i^*\}}}_{e_{lj}\in \mathcal{E}} x_l + e_j\Big|, ~j\in \Gm(i^*) \Big] & \leq \frac{1}{\lceil d/2 \rceil - (d-\gm)} \|\noise_{\Gm'(i^*)}\|_1  \\
& \leq \frac{2}{(1-4\epsilon_{\mathcal{T}_{s}})d} \|\noise_{\Gm'(i^*)}\|_1
\end{split}
\end{equation}

(b) If on the other hand $i\in\mathcal{R} := \support\backslash \{i^*\}$, firstly we rewrite $\displaystyle \Big| \mathop{\sum_{l\in \support\backslash \{i\}}}_{e_{lj}\in \mathcal{E}} x_l + e_{j}\Big|,  ~j\in \Gm(i^*)$ as
\begin{align}
\label{eqn:lmm_lhs_bnds3}
\Big| \mathop{\sum_{l\in \mathcal{R}\backslash \{i\}}}_{e_{lj}\in \mathcal{E}} x_l + \mathbb{I}_{\{e_{i^*j}\in \mathcal{E}\}}x_{i^*} + e_{j}\Big|,  ~j\in \Gm(i^*),
\end{align}
where $\mathbb{I}_{\{\cdot\}}$ is an indicator function defined as follows
\begin{equation*}
\mathbb{I}_{\{i\in\Omega\}} = \begin{cases} 1 &\mbox{if } i\in\Omega,\\
0 & \mbox{if } i\notin\Omega. \end{cases}
\end{equation*}
Since we assume the induction statement to be true for $\support$ with $|\support|=s-1$, we replace $\support$ by $\mathcal{R}$ and $e_j$ by $e'_j = \mathbb{I}_{\{e_{i^*j}\in \mathcal{E}\}}x_{i^*} + e_{j}$ and we use the fact that $\epsilon_{\mathcal{T}_{s-1}} \leq \epsilon_{\mathcal{T}_{s}}$ to get an upper bound from \eqref{eqn:lem_median} as follows:
\begin{equation}
\label{eqn:lmm_lhs_bnds4}
\sum_{i\in \mathcal{R}} q_2 \Big[\Big| \mathop{\sum_{l\in \mathcal{R}\backslash \{i\}}}_{e_{lj}\in \mathcal{E}} x_l + e_j\Big|, ~j\in \Gm(i) \Big] \leq \frac{4\epsilon_{\mathcal{T}_{s}} d}{d-4\epsilon_{\mathcal{T}_{s}} d}\|\x_{\mathcal{R}}\|_1 + \frac{2}{(1-4\epsilon_{\mathcal{T}_{s}})d}\|\noise'_{\Gm(\mathcal{R})}\|_1,
\end{equation}
where $\noise'$ is a vector with components $e'_j$ for $j = 1,\ldots,m$.
Next we upper bound $\|\noise'_{\Gm(\mathcal{R})}\|_1$ as follows
\begin{equation}
\label{eqn:lmm_lhs_bnds5}
\|\noise'_{\Gm(\mathcal{R})}\|_1 \leq \sum_{j\in \Gm(\mathcal{R})} \mathbb{I}_{\{e_{i^*j}\in \mathcal{E}\}} \left|x_{i^*}\right| + \|\noise_{\Gm(\mathcal{R})}\|_1 \leq 2\epsilon_{\mathcal{T}_{s}} d \left|x_{i^*}\right| + \|\noise_{\Gm(\mathcal{R})}\|_1,
\end{equation}
where we used the fact that
\begin{align*}
\sum_{j\in \Gm(\mathcal{R})} \mathbb{I}_{\{e_{i^*j}\in \mathcal{E}\}}  = \sum_{j\in \Gm''(i^*)} \mathbb{I}_{\{e_{i^*j}\in \mathcal{E}\}} = \left|\Gm''(i^*)\right| \leq 2\epsilon_{\mathcal{T}_{s}} d,
\end{align*}
which is the cardinality of the collision set deducible from Proposition \ref{pro:unique_nbr}.

We now bound the last expression in  \eqref{eqn:lmm_lhs2}, i.e., $\sum_{i\in \support}q_2 \Big[\Big| \mathop{\sum_{l\in \support\backslash \{i\}}}_{e_{lj}\in \mathcal{E}} x_l + e_j\Big|, ~j\in \Gm(i) \Big]$, which is the sum of the two cases (a) and (b) leading respectively to the bounds in \eqref{eqn:lmm_lhs_bnds2} and \eqref{eqn:lmm_lhs_bnds4}. This sum can thus be rewritten as follows
\begin{equation}
\label{eqn:lmm_lhs_bnds6}
q_2 \Big[\Big| \mathop{\sum_{l\in \support\backslash \{i^*\}}}_{e_{lj}\in \mathcal{E}} x_l + e_j\Big|, ~j\in \Gm(i^*) \Big]  + \sum_{i\in \mathcal{R}} q_2 \Big[\Big| \mathop{\sum_{l\in \support\backslash \{i\}}}_{e_{lj}\in \mathcal{E}} x_l + e_j\Big|, ~j\in \Gm(i) \Big].
\end{equation}
Now using the bounds in \eqref{eqn:lmm_lhs_bnds2} and \eqref{eqn:lmm_lhs_bnds4} with \eqref{eqn:lmm_lhs_bnds5}, we upper bound \eqref{eqn:lmm_lhs_bnds6} by
\begin{equation*}
\label{eqn:lmm_lhs_bnds7}
\frac{2}{(1-4\epsilon_{\mathcal{T}_{s}})d} \|\noise_{\Gm'(i^*)}\|_1 + \frac{4\epsilon_{\mathcal{T}_{s}} d}{d-4\epsilon_{\mathcal{T}_{s}} d}\|\x_{\mathcal{R}}\|_1
+ \frac{2}{(1-4\epsilon_{\mathcal{T}_{s}})d} \left(2\epsilon_{\mathcal{T}_{s}} d \left|x_{i^*}\right| + \|\noise_{\Gm(\mathcal{R})}\|_1\right),
\end{equation*}
which in turn can be bounded by
\begin{equation}
\label{eqn:lmm_lhs_bnds8}
\frac{4\epsilon_{\mathcal{T}_{s}}}{1-4\epsilon_{\mathcal{T}_{s}}}\|\x_{\support}\|_1 + \frac{2}{(1-4\epsilon_{\mathcal{T}_{s}})d} \|\noise_{\Gm(\support)}\|_1.
\end{equation}

The upper bound \eqref{eqn:lmm_lhs_bnds7} is due to the fact that $\Gm(\mathcal{R})$ and $\Gm'(i^*)$ are two disjoint subsets of $\Gm(\support)$ and this completes the proof.
 \end{proof}

 \subsection{Proof of Theorem \ref{thm:MEIHT_conv}}

\begin{proof} We use Lemma \ref{lem:median} to prove Theorem \ref{thm:MEIHT_conv}. In order to show that \eqref{eqn:MEIHT_error} holds, we show that \eqref{eqn:MEIHT_error1} below holds and by induction this leads to \eqref{eqn:MEIHT_error}.
\begin{equation}
\label{eqn:MEIHT_error1}
\|{\bf x}^{n+1} - {\bf x}_{\mathcal{S}}\|_1 \leq \alpha\|{\bf x}^n - {\bf x}_{\mathcal{S}}\|_1 + (1-\alpha)\beta\|{\bf Ax}_{\bar{\mathcal{S}}} + {\bf e}\|_1.
\end{equation}
Let us define $\mathcal{Q}^{n+1} := \support \cup \supp(\x^n) \cup \supp(\x^{n+1})$ and ${\bf u}^{n+1} := \left(\x^n + \median\left(\obs - \AA\x^n\right)\right)_{\mathcal{Q}^{n+1}}$. Since we are able to do exact projections (in the $\ell_1$ norm) onto the given model, then $\x^{n+1}$ is a better $\mathcal{M}_s$-sparse approximation to ${\bf u}^{n+1}$ than is $\x_{\support}$, i.e.,
$$\|{\bf u}^{n+1} - {\bf x}^{n+1}\|_1 \leq \|{\bf u}^{n+1} - {\bf x}_{\mathcal{S}}\|_1.$$
With this fact we use the triangle inequality to have
\begin{equation}
\label{eqn:MEIHT_error2}
\|{\bf x}^{n+1} - {\bf x}_{\mathcal{S}}\|_1 \leq \|{\bf x}^{n+1} - {\bf u}^{n+1}\|_1 + \|{\bf u}^{n+1} - {\bf x}_{\mathcal{S}}\|_1 \leq 2\|{\bf u}^{n+1} - {\bf x}_{\mathcal{S}}\|_1.
\end{equation}
We substitute the value of ${\bf u}^{n+1}$ and use the fact that $\support \subseteq \mathcal{Q}^{n+1}$ to upper bound \eqref{eqn:MEIHT_error2} by the following
\begin{equation}
\label{eqn:MEIHT_error3}
2\|\left[{\bf x}_{\mathcal{S}} - \x^n - \median\left(\AA\left(\x^n - \x_\support\right) + \bar{\noise}\right)\right]_{\mathcal{Q}^{n+1}}\|_1,
\end{equation}
where $\bar{\noise} := \AA\x_{\bar{\support}} + \noise$. 
Next we apply Lemma \ref{lem:median} to bound \eqref{eqn:MEIHT_error3} by
\begin{equation}
\label{eqn:MEIHT_error4}
\frac{8\epsilon_{\model_{3s}}}{1-4\epsilon_{\model_{3s}}}\|\x_\support - \x^n\|_1 + \frac{4}{\left(1 - 4\epsilon_{\model_{3s}}\right)d}\|\bar{\noise}\|_1,
\end{equation}
where we use $\model_{3s}$ because $\mathcal{Q}^{n+1} \in \model_{3s}$ due to the fact that our models have the {\em nestedness property} (see Section \ref{sec:nested}). This bound is the desired upper bound in \eqref{eqn:MEIHT_error1}, with $\alpha := \frac{8\epsilon_{\model_{3s}}}{\left(1-4\epsilon_{\model_{3s}}\right)}$ and $\beta := \frac{4}{\left(1 - 12\epsilon_{\model_{3s}}\right)d}$.

Then, we use induction on the bound in \eqref{eqn:MEIHT_error1}  to get the following bound
\begin{align*}
\label{eqn:MEIHT_error5}
\|{\bf x}^{n} - {\bf x}_{\mathcal{S}}\|_1 \leq \alpha^n\|{\bf x}^0 - {\bf x}_{\mathcal{S}}\|_1 + \sum_{k=0}^{n-1}\alpha^k (1-\alpha^n)\beta\|{\bf Ax}_{\bar{\mathcal{S}}} + {\bf e}\|_1.
\end{align*}
Simplifying the geometric series leads to \eqref{eqn:MEIHT_error}. To have $\alpha<1$ and $\beta>0$ we need $\epsilon_{\model_{3s}} < 1/12$. This concludes the proof.
\end{proof}

\subsection{Proof of Corollary \ref{cor:MEIHT_conv}}

\begin{proof}
Let the sequence $(\x^n)$ for $n\geq 0$ converges to $\widehat{\x}$. But, as $n\rightarrow \infty$, $\alpha^n \rightarrow 0$, the first term in \eqref{eqn:MEIHT_error} goes to zero, and hence, the term multiplying $\beta$ in the second terms becomes 1. Therefore, by the triangle inequality we have the following
$$
\|\widehat{\x} - {\bf x}_{\mathcal{S}}\|_1  \leq \beta\|\AA\x_{\bar{\support}}\|_1 +  \beta\|\noise\|_1.
$$
Note that $\|\AA\x_{\bar{\support}}\|_1 \leq d\|\x_{\bar{\support}}\|_1$ holds irrespective of $\x_{\bar{\support}}$ being sparse or not. This is a direct consequence of expanders and an indirect one of model RIP-1 matrices. Therefore, we bound our error as thus
\begin{align*}
\|\widehat{\x} - {\bf x}\|_1 \leq \|\x_{\bar{\support}}\|_1  +  \|\widehat{\x} - {\bf x}_\support\|_1 \leq (1+\beta d)\sigma_{\model_s}(\x)_1 + \beta\|\noise\|_1.
\end{align*}
Taking $C_1 = 1+\beta d$ and $C_2 = \beta$ completes the proof.
\end{proof}


\section{Appendix}\label{sec:appdx}

\subsection*{Proof of Lemma \ref{lem:median} for $\GG_s$-sparse models} \label{sec:approof}

Note that here $s$ is the model order and for $\mathcal{S}\in\GG_s$, $\|\x_{\mathcal{S}}\|_0\leq kg_{\max}$ where $k$ is the within group sparsity and the $g_{\max}$ is the maximum size of the groups. As in the proof for $\mathcal{T}_s$ models in Section \ref{sec:proof}, we will arrive at the bound in the lemma by bounding the right hand side of the following inequality, which is the same as \eqref{eqn:lmm_lhs2}, for all $\mathcal{S}\in\GG_s$.
 \begin{equation}
 \label{eqn:lmm_lhs2_mod}
\|\left[ \mathcal{M} \left({\bf Ax}_{\mathcal{S}} + {\bf e}\right) - {\bf x} \right]_{\support}\|_1 \leq \sum_{i\in \support}q_2 \Big[\Big| \mathop{\sum_{l\in \support\backslash \{i\}}}_{e_{lj}\in \mathcal{E}} x_l + e_j\Big|, ~j\in \Gm(i) \Big],
\end{equation}
Similar to the proof for $\mathcal{T}_s$ models we proceed in an inductive manner. 

{\bf Base case:} When $\mathcal{S}$ has model order $|\support| = s = 1$, the sparsity of $\mathcal{S} = g_{\max}$. In this case we just take the support as the generic $g_{\max}$-sparse vector and whose proof is equivalent to the of $\mathcal{T}_s$ models (with $s=g_{\max}$) in Section \ref{sec:proof} above.

{\bf Inductive step:} We assume that the induction statement holds for $\support$ with model order $|\support| = s-1$ ($s\geq 2$) which means $\x_{\mathcal{S}}$ has sparsity $(k-1)g_{\max}$. Then attempt to show that it also holds for $\support$ with model order $|\support| = s$ where $\x_{\mathcal{S}}$ has sparsity $kg_{\max}$. Note that the difference in the sparsity of a model order of $s-1$ and $s$ is $g_{\max}$ nodes not $1$ node. Let $\mathcal{S}'$ be the model of order $s-1$, $\mathcal{R}\subset \mathcal{S}$ and $i^*$ be a {\em unique neighbour} index. Then we split $\mathcal{S}$ into disjoint sets such that $\mathcal{S} = \mathcal{S}'\cup\mathcal{R}\cup\{i^*\}$, where $\x_{\mathcal{S}'}$, $\x_{\mathcal{R}}$ and $\x_{\{i^*\}}$ are all model sparse by Definition \ref{def:mdlsparse}. If instead $i^* \in \mathcal{S}'$ we will write $\mathcal{S} = \mathcal{S}''\cup\mathcal{R}\cup\{i^*\}$ where $\mathcal{S}' = \mathcal{S}''\cup\{i^*\}$ and the rest of the argument below goes through. Note that it is also possible to consider $\mathcal{S}$ as an $sg_{\max}$-sparse set and then do the induction from $sg_{\max}-1$-sparse to $sg_{\max}$-sparse sets because of the way our models are nested.

For $i=i^*$, using Proposition \ref{pro:unique_nbr} we bound  \eqref{eqn:lmm_lhs2_mod} by a bound on the mean over the cardinality of the collision set, $\Gm''(i^*)$, to get
\begin{equation}
\label{eqn:lmm_lhs_bnds2mod}
q_2 \Big[\Big| \mathop{\sum_{l\in \support\backslash \{i^*\}}}_{e_{lj}\in \mathcal{E}} x_l + e_j\Big|, ~j\in \Gm(i^*) \Big] \leq \frac{2}{(1-4\epsilon_{\GG_{s}})d} \|\noise_{\Gm'(i^*)}\|_1,
\end{equation}
where $\epsilon_{\GG_{s}}$ is an upper bound on the expansion coefficient of $\{i^*\}$.
For $\mathcal{S}'\cup\mathcal{R}$, we therefore have the following: 
\begin{equation}
\label{eqn:lmm_lhs_bndsSpr}
\sum_{i\in \mathcal{S}'\cup\mathcal{R}} q_2 \Big[\Big| \mathop{\sum_{l\in \mathcal{S}\backslash \{i\}}}_{e_{lj}\in \mathcal{E}} x_l + e_j\Big|, ~j\in \Gm(i) \Big] 
\leq\sum_{i\in \mathcal{S}'\cup\mathcal{R}} q_2 \Big[\Big| \mathop{\sum_{l\in \mathcal{S}'\cup\mathcal{R}\backslash \{i\}}}_{e_{lj}\in \mathcal{E}} x_l + e'_{j}\Big|,  ~j\in \Gm(i)\Big],
\end{equation}
where $e'_j = \mathbb{I}_{\{e_{i^*j}\in \mathcal{E}\}}x_{i^*} + e_{j}$. Now we use the fact that $\epsilon_{\GG_{s}}$ upper bounds the expansion coefficient of $\mathcal{S}'\cup\mathcal{R}$ combined with \eqref{eqn:lem_median} to get the following upper bound
\begin{equation}
\label{eqn:lmm_lhs_bndsSpr2}
\frac{4\epsilon_{\GG_{s}} d}{d-4\epsilon_{\GG_{s}} d}\|\x_{\mathcal{S}'\cup\mathcal{R}}\|_1 + \frac{2}{(1-4\epsilon_{\GG_{s}})d}\|\noise'_{\Gm(\mathcal{S}'\cup\mathcal{R})}\|_1.
\end{equation}
But $\|\noise'_{\Gm(\mathcal{S}'\cup\mathcal{R})}\|_1 \leq 2\epsilon_{\GG_{s}} d \left|x_{i^*}\right| + \|\noise_{\Gm(\mathcal{S}'\cup\mathcal{R})}\|_1$ and substituting this bound into the sum of \eqref{eqn:lmm_lhs_bnds2mod} and \eqref{eqn:lmm_lhs_bndsSpr2} simplifies to:
\begin{equation}
\label{eqn:lmm_lhs_bnds7}
\frac{4\epsilon_{\GG_{s}}}{1-4\epsilon_{\GG_{s}}}\|\x_{\support}\|_1 + \frac{2}{(1-4\epsilon_{\GG_{s}})d} \|\noise_{\Gm(\support)}\|_1.
\end{equation}

The upper bound \eqref{eqn:lmm_lhs_bnds7} is due to the fact that $\Gm(\mathcal{S}'\cup\mathcal{R})$ and $\Gm'(i^*)$ are two disjoint subsets of $\Gm(\support)$ and the separability of the $\ell_1$-norm. We also implicitly assume the nestedness property which our models satisfy. This completes the proof.

\end{document}